\documentclass[10pt]{amsart}
\usepackage{amssymb}
\usepackage[colorlinks]{hyperref}

\DeclareMathOperator{\Coz}{Coz}
\DeclareMathOperator{\coz}{coz} 
\DeclareMathOperator{\ann}{Ann} 
\DeclareMathOperator{\ma}{Max}

\theoremstyle{plain}
\newtheorem{theorem}{Theorem}[section]
\newtheorem{prop}[theorem] {Proposition}
\newtheorem{lemma}{Lemma}[section]
\newtheorem {corol}{Corollary}[theorem]

\theoremstyle{definition}
\newtheorem{definition}{Definition}

\newtheorem{remark}{\textup{Remark}} 
\newtheorem{example}{\textit{Example}} 
\newtheorem*{acknowledgement}{\textup{Acknowledgement}}

\begin{document}
	
	\title[Maximal ideals in $\mathcal{R}_cL$]%
	{Maximal Ideals in Functions Rings with a Countable Pointfree Image}
	\author[M. Abedi]
	{Mostafa Abedi\\
		Esfarayen University of Technology, Esfarayen, North Khorasan, Iran\\
		Email Address: {ms$_-$abedi@yahoo.com, abedi@esfarayen.ac.ir}}

	

	
	
	
	\begin{abstract}
		Consider the subring $\mathcal{R}_cL$ of continuous real-valued functions defined on a frame $L$, comprising  functions with a countable pointfree image. We present some useful properties of $\mathcal{R}_cL$. We establish that both $\mathcal{R}_cL$ and its bounded part, $\mathcal{R}_c^*L$, are clean rings for any frame $L$.  We show that, for any completely regular frame $L$, the $z_c$-ideals of $\mathcal{R}_cL$ are contractions of the $z$-ideals of $\mathcal{R}L$. This leads to the conclusion  that maximal ideals (or prime $z_c$-ideals) of $\mathcal{R}_cL$   correspond precisely to the contractions of those of $\mathcal{R}L$. 
		We introduce the ${\bf O}_c$- and ${\bf M}_c$-ideals of $\mathcal{R}_cL$. 
		By using ${\bf M}_c$-ideals, we characterize the maximal ideals of $\mathcal{R}_cL$, drawing an analogy with the Gelfand-Kolmogoroff theorem for the maximal ideals of $C_c(X)$. 
		We demonstrate that  fixed maximal ideals of $\mathcal{R}_cL$ have a one-to-one correspondence with the points of $L$  in the case where $L$ is a zero-dimensional frame. We describe the maximal ideals of $\mathcal{R}_c^*L$, leading to  a one-to-one correspondence between these ideals and the points of $\beta L$, the Stone-\v{C}ech compactification of $L$, when $L$ is a strongly zero-dimensional frame. 
		Finally, we establish that $\beta_0L$, the Banaschewski compactification of a zero-dimensional $L$, is isomorphic to the frames of the structure spaces of $\mathcal{R}_cL$, $\mathcal{R}_c(\beta_0L)$, and $\mathcal{R}(\beta_0L)$.
		\\
		
		
	\end{abstract}
	
	\maketitle
	
\section{Introduction}\label{sec:intro}
Consider the ring $C(X)$, comprising  all continuous real-valued  functions on a topological space $X$ (refer to \cite{gl} for more details). The subring $C_c(X)$ of $C(X)$, comprised of functions with a countable image, is an $\mathbb{R}$-subalgebra of $C(X)$.
The ring $C_c(X)$ is introduced and studied in \cite{gkn, gkn1}. The maximal ideals of $C_c(X)$ and its bounded part $C_c^*(X)$ are examined in \cite{akko}. In that paper, the authors characterize the maximal ideals of $C_c(X)$ based on the Gelfand-Kolmogoroff theorem for the maximal ideals of $C(X)$ and provide two representations of the maximal ideals of $C_c^*(X)$. They prove that both $C_c(X)$ and $C_c^*(X)$ are always clean rings for any space $X$. 

Let ${\bf Top}$ designate the category of topological spaces with continuous functions, and let ${\bf Frm}$ designate the category of frames with frame maps. Denote by $\mathfrak{O}X$ the frame of open subsets of a topological space $X$.
Any continuous function $f: X\to Y$ between spaces gives rise to a frame map $\mathfrak{O}f : \mathfrak{O}Y\to\mathfrak{O}X$ with $\mathfrak{O}f(V)=f^{-1}(V)$ for  any open set $V$ of $Y$. 
 For any space $X$, there is a one-one onto map  
\[{\bf Frm}(\mathcal{L}(\mathbb{R}), \mathfrak{O}X))\to{\bf Top(X, \mathbb{R})}
\]
given by correspondence $\alpha\mapsto\tilde{\alpha}$ such that
\[
p<\tilde{\alpha}(x)<q \text{ iff } x\in\alpha(p, q)
\]
whenever $p<q$ in $\mathbb{Q}$ (refer to  \cite{bana} for more details). 
Suppose that $f\in C(X)$ and $r_0\in\mathbb{R}$. As usual, let $Z(f)=\{x\in X \mid f(x)=0\}=f^{-1}(0)$ and $\Coz(f)=X-Z(f)$. Now, we have
\[
r_0\in f[X] \Leftrightarrow Z(f-r_0)\ne\emptyset\Leftrightarrow\Coz(f-r_0)\ne X
\]
So, we can infer that
\[
f\in C_c(X)\Leftrightarrow \text{ the set } \{r\in\mathbb{R} \mid \Coz(f-r)\ne X\} \text{ is countable}.
\]

Consider the ring $\mathcal{R}L$, consisting of continuous real-valued functions defined on a frame $L$. Given any $\alpha\in\mathcal{R}L$, we define $R_{\alpha}=\{ r\in\mathbb{R} \mid  \coz(\alpha-{\bf r})\ne 1\}$, as specified in \cite{ese}. Obviously,  $R_{\alpha}$ is the extension to arbitrary $\mathcal{R}L$ of the familiar correspondence between functions on spaces and their images. Indeed, if the continuous map $g : X\to\mathbb{R}$ corresponds to $\alpha: \mathcal{L}(\mathbb{R}) \to \mathfrak{O}X$ such that $g=\tilde{\alpha}$ as in the previous paragraph, then, for any $r\in\mathbb{R}$, we have 
\[\begin{aligned}
	r \in g[X] &\Leftrightarrow\Coz(g-r)\ne X\\
	      &\Leftrightarrow\{x\in X \mid g(x)<r\}\cup \{x\in X \mid g(x)>r)\}\ne X\\
	      &\Leftrightarrow g^{-1}[(-,r)]\cup g^{-1}[(r,-)]\ne X\\
	      &\Leftrightarrow \alpha(-,r)\cup\alpha(r,-)\ne X \\
	      &\Leftrightarrow \coz(\alpha-{\bf r})\ne X \\
	      &\Leftrightarrow r\in R_{\alpha}.
\end{aligned}\]


We say that an element $\alpha$ of $\mathcal{R}L$ has a \emph{  countable pointfree  image } if $R_\alpha$ forms a countable set. The subring of  $\mathcal{R}L$ consisting of those functions with a  countable pointfree image, denoted by $\mathcal{R}_cL$, can be considered as a pointfree topology version of  $C_c(X)$.  
The ring $\mathcal{R}_cL$ is introduced and studied under a different symbol, $\mathcal{C}_c(L)$,  in \cite{tes}. In that paper, the  authors demonstrate that  $\mathcal{R}_c(\mathfrak{O}X)\cong C_c(X)$ for any topological space $X$.

It is important to note that in \cite{kes}, the authors present and analyze a subring of $\mathcal{R}L$ as  a pointfree topology version of  $C_c(X)$, referring to this ring as $\mathcal{R}_cL$. The version introduced in \cite {tes} is better suited for functions with countable ranges. In fact,  Proposition 3.3 in \cite{ese} shows that the  version introduced in \cite{kes} is a subring of the specified version in \cite{tes}. Throughout, we adopt the version introduced in \cite {tes} and denote it as $\mathcal{R}_cL$.

 Recall that a \emph{zero-dimensional} space is a space whose clopen subset form a base for its topology. In \cite[Theorem 4.6]{gkn}, it is shown that for any space $X$, there exists a zero-dimensional space $Y$
such that $C_c(X)\cong C_c(Y)$. The counterpart of this theorem
 can be found in \cite[Proposition 7.1]{tes}. This proposition states that for any frame $L$, there exists a zero-dimensional frame $M$ such that $\mathcal{R}_cL\cong \mathcal{R}_cM$. Hence, when studying the rings $\mathcal{R}_cL$, one might assume that all frames  are zero-dimensional. However, we choose not to make this assumption because  some results  should be  discussed without assuming zero-dimensionality.

Now, let us provide a summary of the article's content. Section \ref{2} contains definitions, results, and terminology  used throughout the article. 
Some ring-theoretic and frame-theoretic  properties of $\mathcal{R}_cL$, which were previously undescribed, are presented in Section \ref{3}. We observe that a completely regular frame $L$ is  strongly zero-dimensional  iff 
$ \Coz L=\Coz_cL$  iff 
$\beta L\cong \beta_0L$ (Theorem \ref{cs}). We show that both $\mathcal{R}_cL$ and  $\mathcal{R}_c^*L$ are always clean (Theorem \ref{ac}). We show that the  prime $z_c$-ideals and the maximal ideals of $\mathcal{R}_cL$ are the contractions of the corresponding ideals of $\mathcal{R}L$ (Corollaries \ref{pc} and \ref{mc}). Finally, we present some characterizations of $d$-ideals in terms of the cozero elements (Proposition \ref{d} and Theorem \ref{d1}).

In Section \ref{4}, inspired by \cite[Example 4.11]{du0} and the definition of ${\bf O}_c$- and ${\bf M}_c$-ideals of $C_c(X)$, we introduce the ${\bf O}_c$- and ${\bf M}_c$-ideals of $\mathcal{R}_cL$. Initially, we survey  some useful properties of the right adjoint of the join map $\beta_0L\to L$ (Lemma \ref{c1} and Proposition \ref{aa}). Let $I$ be an element in $\beta_0L$. 
It is observed that the ideal ${\bf M}^I_c$ is a $z_c$-ideal, and the ideal ${\bf O}^I_c$ is a $d$-ideal which is generated by a set of idempotents  in $\mathcal{R}_cL$ (Example \ref{mo} and Theorem \ref{ig}). It is shown that ${\bf M}^I_c$ ideals are distinct for distinct elements in $\beta_0L$ (Lemma \ref{i}).

In Section \ref{5}, we provide a comprehensive description of the maximal ideals of $\mathcal{R}_cL$ similar to the maximal ideals of  $C_c(X)$. Theorem \ref{m} presents the counterpart of the Gelfand-Kolmogoroff theorem of $C_c(X)$. Additionally, Proposition \ref{m1} shows the locations of all the ideals of $\mathcal{R}_cL$, which essentially is a $\mathcal{R}_cL$ version of  McKnight's Lemma. This lemma states where all the ideals of $C(X)$ lie (reported in \cite{koh} to have been shown by J. D. McKnight, Jr. in his 1953 unpublished purdue doctoral thesis).


Fixed maximal ideals of $\mathcal{R}_cL$ are examined in Section \ref{6}. It is demonstrated that these ideals have a one-to-one correspondence with prime elements of $L$ (Theorem \ref{ff}). Furthermore, we establish that a zero-dimensional frame $L$ is compact  iff every maximal ideal of $\mathcal{R}_cL$ is  fixed, and this condition is also equivalent to every maximal ideal of $\mathcal{R}_c^*L$ being  fixed.

Section \ref{7} is dedicated to maximal ideals in $\mathcal{R}_c^*L$.  We show that the rings $\mathcal{R}_c(\beta_0L) $ and $\mathcal{R}_c^*L$ are isomorphic when  the quotient map $\bigvee: \beta_0L \to L$ is a $C_c^*$-quotient (Lemma \ref{cs1}). Additionally, we observe that $L$ is not always  a $C_c^*$-quotient of $\beta_0L$ (Proposition \ref{qn}). 
In Proposition \ref{acc}, we establish that absolutely convex ideals of $\mathcal{R}_c^*L$ are  contractions of the corresponding ideals of $\mathcal{R}^*L$. Using this finding, we provide two representations of maximal ideals of $\mathcal{R}^*L$ (Theorem \ref{tw}). 

In the closing section, it is shown that there exists an isomorphism between $\beta_0L$ and the frames of the structure spaces of $\mathcal{R}_cL$, $\mathcal{R}_c(\beta_0L)$, and $\mathcal{R}(\beta_0L)$ (Theorem \ref{b} and Corollary \ref{bb}).
\section{Preliminaries}\label{2}
\subsection{Frames and their Homomorphisms}
For a general framework of frames, refer to \cite{j, pp2012}. Throughout, $L$ stands for a frame.
The  bottom and top elements of $L$ are denoted by $0_L$ and $1_L$ respectively. If $L$ is clear from the context, we drop the subscripts. 

Suppose $a$ and $b$ are elements of a frame $L$. The \emph{relative pseudocomplement} of $a$ with respect to $b$, denoted as $a\longrightarrow b$, is defined as
$a\longrightarrow b=\bigvee\big\{x\in L \mid a\wedge x\leq b\big\}$.  The \emph{  pseudocomplement} of $a$, denoted as $a^*$, is given by 
$a^*=a\longrightarrow0$.
If $a^*\vee b=1$, then $a$ is said to be \emph{  rather below} $b$, denoted as $a\prec b$. Moreover, if there exists a sequence $\{a_q \mid q\in Q\cap[0,1]\}$ such that $a=a_0$, $b=a_1$ and  $a_p\prec a_q$ whenever $p<q$, then $a$ is considered to be \emph{  completely below} $b$, denoted as $a\prec\!\!\prec b$.  
If $a\vee a^*=1$, then $a$ is called a \emph{ complemented} element in $L$. We denote the set of all complemented elements of $L$ by $BL=\{a\in L  \mid  a\vee a^*=1\}$. A frame $L$ is called \emph{regular}, \emph{completely regular}, and \emph{zero-dimensional} if fore every $a\in L$ we have $a=\bigvee\{x\in L : x\prec a\}$, $a=\bigvee\{x\in L : x\prec\!\!\prec a\}$, and  $a=\bigvee \{x\in BL : x\leq a\}$, respectively.  

A frame $L$ is called \emph{  compact} if, whenever
$1=\bigvee T$ for $T\subseteq L$, then $1=\bigvee S$ for some finite subset $S\subseteq T$. 
For each $a\in L$, the sets $\downarrow\!a=\{x\in L  \mid  x\leq a\}$ and $\uparrow\!a=\{x\in L  \mid  a\leq x\}$ form frames. 
An element $1\ne p$ of $L$ is referred to as \emph{point} (or \emph{prime}) if $a\wedge b\leq p$ implies $a\leq p$ or $b\leq p$. We use $Pt(L)$ to denote the  set of all prime elements of $L$. The primes of any regular frame are exactly its maximal elements.

A \emph{frame homomorphism}, also known as a \emph{frame map}, is a mapping between frames that preserves finite meets (including 1) and arbitrary joins (including 0). 
A frame homomorphism $f : L\to M$ is \emph{dense} when $f(x)=0_M$ implies $x=0_L$. The \emph{quotient} of a frame $L$ refer to an onto homomorphism $h  :  L\to M$, where $h$ is known as the  \emph{quotient map}.
A \emph{ compactification} of $L$ is a dense onto homomorphism $h:M\rightarrow L$ with a compact regular domain.

In any bounded distributive lattice $A$, an \emph{  ideal} is a subset $I\subseteq A$ satisfying two conditions: for any finite $J\subseteq I$, their join $\bigvee J $ is in $I$, and if $x\leq y$ with  $y\in I$, then $x$ is also in $I$.
The set $Id(A)$ of all ideals of $A$ forms a frame when it is ordered by inclusion. In this frame, the ideal generated by $\bigcup I_{\lambda}$ as $\bigvee I_{\lambda}$, and $\bigwedge I_{\lambda}=\bigcap I_{\lambda}$.


\subsection{The Stone-\v{C}ech compactification}
An ideal $I$ of a frame $L$ is said to be \emph{  completely regular} if, for any $a\in I$, there exists $b\in I$ such that $a\prec\!\!\prec b$. The frame of completely regular ideals of $L$ is regarded as the Stone-\v{C}ech compactification of $L$, denoted $\beta L$ (refer to \cite{ba}). We use $r$ to represent the right adjoint of the join map $j_L : \beta L \to L$. To recall, $r(a)=\{x\in L \mid x\prec\!\!\prec a\}$ for any $a\in L$.
 

\subsection{The Banaschewski compactification}
For any frame $L$, the frame of all ideals of $BL$ is denoted as $\beta_0L$. It is important to note that $\beta_0L$ is a compact zero-dimensional frame. Notably,  the join map $j_0 : \beta_0L\to L$ is indeed a dense frame homomorphism. Moreover, the join map $j_0:\beta_0L\rightarrow L$ is a
compactification for a frame $L$  iff $L$ is zero-dimensional (refer to \cite{bana111}). The right adjoint of the join map $j_0 : \beta_0L\to L$ is denoted by $r_0$. To recall, $r_{0}(a)=\{x\in BL \mid x\leq a\}=\downarrow\!a\cap BL$ for any $a\in L$.
\subsection{The ring $\mathcal{R}L$ and the cozero part of $L$}
Consider $\mathcal{L}(\mathbb{R})$, denoting the frame of reals, which is isomorphic to the frame $\mathfrak{O}\mathbb{R}$
(refer to \cite{bana}). The ring $\mathcal{R}L$ is comprised of frame homomorphisms $ \mathcal{L}(\mathbb{R})\to L$, and it is a reduced $f$-ring (as discussed \cite{ball2002, bana}).
We adopt the notation from \cite{bana}. We represent the bounded part of $\mathcal{R}L$ as  $\mathcal{R}^*L$, and use ${\bf0}$ and ${\bf1}$ to denote the zero element and the identity element of $\mathcal{R}L$, respectively. For detailed properties of the cozero map $\coz : \mathcal{R}L\to L$, we refer to \cite{ball2002} and \cite {bana}. To recall, for each $\varphi\in\mathcal{R}L$, 
$
\coz\varphi=\varphi(-,0)\vee\varphi(0,-),
$
where
\[
(-,0)=\bigvee\{(p, 0) \mid 0>p\in\mathbb{Q}\}\quad\mbox{ and }\quad(0,-)=\bigvee\{(0, q) \mid 0<q\in\mathbb{Q}\}.
\]
The set $\Coz L=\{\coz\varphi \mid \varphi\in\mathcal{R}L\}$, known as the \emph{  cozero part} of $L$, is a sublattice of $L$ for any frame $L$, and that $L$ is completely regular iff it is generated by $\Coz L$.
\subsection{The ring $\mathcal{R}_{c}L$}
The ring $\mathcal{R}_{c}L=\{\alpha\in\mathcal{R}L \mid R_{\alpha} \mbox{ is countable}  \}$ is a $f$-ring, and we denote its bounded part as $\mathcal{R}_{c}^*(L)$, defined as, $\mathcal{R}_{c}^*(L)=\mathcal{R}_{c}L\cap\mathcal{R}^*L$ (refer to  \cite{tes} for more information). The set
$\Coz _cL=\{\coz\varphi \mid \varphi\in\mathcal{R}_{c}L\}$ is a sublattice of $L$ containing $0$ and $1$, for any frame $L$, and that $L$ is a zero-dimensional frame iff it is generated by $\Coz_c L$ (see \cite{et}.
  It is shown that, in \cite{et}, every element of $\Coz _cL$ is a countable join of complemented elements of $L$, for any frame $L$, and that  $\Coz _cL$ is  closed under countable joints when $L$ is a completely regular frame.
We observe that $BL\subseteq \Coz _cL$ for any frame $L$.
In \cite{a} and  \cite{ae}, the following lemma is presented in various forms.
\begin{lemma}\label{cb}
Suppose $\coz(\varphi)\prec\!\!\prec\coz(\delta)$ for some $\varphi, \delta\in\mathcal{R}_c(L)$. Then there exists an element $\rho\in\mathcal{R}_c(L)$ such that $\varphi=\rho\delta$.	
\end{lemma}
		\section{Some properties of $\mathcal{R}_cL$}\label{3}
		Recall that a frame is defined to be \emph{  strongly zero-dimensional} if its Stone-\v{C}ech compactification is zero-dimensional.
		In \cite{bb}, it is shown that a completely regular frame $L$ is  strongly zero-dimensional if, for any $a,b\in L$ with $a\prec\!\!\prec b$, there exits  $c\in BL$ such that $a\leq c\leq b$.
		
		\begin{theorem} \label{cs}
			The following statements are equivalent for each completely regular frame $L$.
			\begin{enumerate} 
				\item $L$ is a strongly zero-dimensional frame.
				\item $ \Coz L=\Coz_cL$.
				\item  $\beta L\cong \beta_0L$.
			\end{enumerate}
		\end{theorem}
		\begin{proof} 
			The equivalence of (1) and (2) can be inferred from the equivalence of (1) and (2) stated in \cite[Proposition 5.5]{du}.
			Clearly, (3) implies (1) since $\beta_0L$ is a zero-dimensional frame. To show that (1) implies (3), we 	
			define a map $f: \beta_0L \to \beta L$ by  $f(I)=\downarrow\!I=\{ x\in L \mid x\leq a \mbox{ for some } a\in I\}$.
			Clearly, $f(I)\in\beta L$ and $f$ preserves order. Also,  $f(0_{\beta_0 L})=\{0_L\}=0_{\beta L}$ and $f(1_{\beta_0 L})=L=1_{\beta L}$. 
			If  $I$ and $J$ belong to $\beta_0L$ and we take $x\in f(I)\wedge f(J)$, then there are elements $x_1$ in $f(I)$ and $x_2$ in $f(J)$ such that $x=x_1\wedge x_2$. This implies we can find $a\in I$ and $b\in J$ with $x\leq a\wedge b$. This shows that $x\in f(I\wedge J)$.  As a result, $f(I)\wedge f(J)\leq f(I\wedge J)$, and so we have equality. Consequently, $f$  preserves finite meets.
			
			Next, suppose we have a collection $\{I_\lambda\}_{\lambda\in\Lambda}$ consisting of elements of $\beta_0L$, and take $x\in f\big(\bigvee\limits_{\lambda\in\Lambda}I_\lambda\big)$. Then $x\leq a$ with $a\in\bigvee\limits_{\lambda\in\Lambda}I_\lambda$. Thus, there are elements $a_{\lambda_1} \in I_{\lambda_1}, \cdots, a_{\lambda_n} \in I_{\lambda_n}$ such that $x\leq a_{\lambda_1}\vee\cdots\vee a_{\lambda_n}$, 
			implying that $x\in\bigvee\limits_{\lambda\in\Lambda}f(I_\lambda)$. As a result,
			$f\big(\bigvee\limits_{\lambda\in\Lambda}I_\lambda\big)\leq \bigvee\limits_{\lambda\in\Lambda}f(I_\lambda)$, and so we have equality since $f$ preserves order. Therefore,  $f$ preserves joins. Consequently, $f$ is a frame homomorphism. 
			Additionally, it is clear that $f$ is dense, and being one-to-one follows from the fact that $\beta L$ and $\beta_0L$ are zero-dimensional. Finally, we need to show that $f$ is onto. For any $J\in\beta L$, let $I=J\cap BL$. It is evident that $I\in\beta_0L$ and $f(I)\subseteq J$. If $x\in J$, then  since $J$ is a completely regular ideal, there is an element $a\in J$ such that $x\prec\!\!\prec a$. By strong zero-dimensionality, we can find $c\in BL$ such that $x\leq c\leq a$, implying that $x\in\downarrow(J\cap BL)=f(I)$.  Consequently, $J\subseteq f(I)$, which completes the proof.
		\end{proof}
		
		Before presenting our next observation, let us recall that  in a commutative ring with identity, an element is considered \emph{  clean} if it can be expressed as the sum of an invertible element and an idempotent. A ring itself is referred to as clean if all its elements are clean. In view of \cite[Proposition 5.3]{du} and \cite[Proposition 4.6]{et1}, a frame $L$ is strongly zero-dimensional  iff $ \mathcal{R}L$ is clean iff $\mathcal{R}^*L$ is clean. By employing a proof  similar to  those of the previously mentioned results, we can establish that $\mathcal{R}_cL$ is always a clean ring. Unfortunately, this approach does  not apply to proving that $\mathcal{R}_c^*L$ is also a clean ring. Fortunately, the following lemma, which is counterpart of \cite[Proposition 3.2]{et1}, promptly shows that the rings $\mathcal{R}_cL$  and $\mathcal{R}_c^*L$ are always clean for any frame $L$.
		
		Let $a\in BL$. Then, for any $p, q\in\mathbb{Q}$, define
		\[e_a(p,q)=\left \{
		\begin{array}{lll}
			0& \hspace{3mm} \mbox{if $p<q\leq 0$ or $1\leq p<q$}\\[2mm]
			a^*&\hspace{3mm} \mbox{if $p<0<q\leq 1$}\\
			a&\hspace{3mm} \mbox{if $0\leq p< 1< q$}\\
			1&\hspace{3mm} \mbox{if $p<0< 1< q$}.
		\end{array}
		\right.
		\]
		By \cite[8.4]{ball2002}, we have $ e_a\in \mathcal{R}L$ such that $\coz( e_a)=a$, 
		$\coz( e_a-{\bf1})=a^*$, and  $e{_a}^2= e_a$. since $R_{ e_a}\subseteq\{0,1\}$, we conclude that $ e_a$   belongs to $\mathcal{R}_cL$. 
		
		Let us  remind the reader that, for every $r\in\mathbb{R}$,  
		we set
		$$(-, r):= \bigvee \limits_{\substack{q\in \mathbb Q\\q<r }} 
		(-,q) \qquad\text{ and }\qquad	(r,-):= \bigvee \limits_{\substack{p\in \mathbb Q\\r<p }} (p.-).$$
		So, for every $(r,\alpha)\in\mathbb{R}\times\mathcal{R}L$, we have 
		$$\alpha(-, r):= \bigvee \limits_{\substack{q\in \mathbb Q\\q<r }} 
		\alpha(-,q), \qquad\text{ and }\qquad
		\alpha(r,-):= \bigvee \limits_{\substack{p\in \mathbb Q\\r<p }} \alpha(p.-),$$
		which implies $ \alpha(-,r)\wedge\alpha(r,-)=0$.
		
Recall from the remark of \cite[Ch. XIV, 5.3.4]{pp2012} that if $\alpha, \beta\in\mathcal{R}L$, then, for every $p\in\mathbb{Q}$, we have
\[
(\alpha-\beta)(p, -)=\bigvee_{q\in\mathbb{Q}}\alpha(q,-)\wedge\beta(-,q-p) \quad\text{ and }\quad
(\alpha-\beta)(-, p)=\bigvee_{q\in\mathbb{Q}}\alpha(-,q)\wedge\beta(q-p,-).
\]
So, for every $(p, r,\alpha)\in\mathbb{Q}\times\mathbb{R}\times\mathcal{R}L$, we have
\[
(\alpha-{\bf r})(p, -)=\bigvee_{q\in\mathbb{Q}}\alpha(q,-)\wedge{\bf r}(-,q-p)=\bigvee_{\substack{q\in \mathbb Q\\p+r<q }}\alpha(q,-)=\alpha(p+r,-)
\]
and
\[
(\alpha-{\bf r})(-, p)=\bigvee_{q\in\mathbb{Q}}\alpha(-,q)\wedge{\bf r}(q-p,-)=\bigvee_{\substack{q\in \mathbb Q\\q<p+r }}\alpha(-,q)=\alpha(-,p+r).
\]
\begin{lemma}\label{al}
	Suppose 
	$(r,\alpha)\in\mathbb{R}\times\mathcal{R}L$ such that $\coz(\alpha-{\bf r})=1$. Then $ \alpha(-,r)$ is a complemented element in $L$ with $\big(\alpha(-,r)\big)^*=\alpha(r,-)$.
\end{lemma}
\begin{proof}
	Since $\coz(\alpha-{\bf r})=(\alpha-{\bf r})(-,0)\vee(\alpha-{\bf r})(0,-)=\alpha(-,r)\vee\alpha(r,-)$,  we have 	$\alpha(-,r)\vee\alpha(r,-)=1$ by hypothesis. On the other hand, $ \alpha(-,r)\wedge\alpha(r,-)=0$ is generally true. Hence, $\alpha(-,r)\vee\alpha(r,-)=1$ and $ \alpha(-,r)\wedge\alpha(r,-)=0,$
	meaning that $c= \alpha(-,r)$ is a complemented element in $L$ with $c^*=\alpha(r,-)$.  
\end{proof}

In the proof of the next lemma, we shall use the following facts:
		
		{\bf A:} \cite[Ch. III, 3.1.3]{pp2012} If $a$ and $b$ are two elements of a frame $L$, then $a\wedge (a\longrightarrow b)=a\wedge b$. 
		
		{\bf B:} \cite[ Lemma 3.1]{et1} If  $a\in BL$ and $\alpha\in\mathcal{R}L$, then the map $\varphi : \mathcal{L}(\mathbb{R})\to L$, given by 
		\[
			\varphi(p,q)=[a\vee\alpha(p,q)]\wedge[a\longrightarrow\alpha(p+1,q+1)]    
			\]
	for every $p,q\in\mathbb{Q}$, is a frame map.
	
	{\bf C:} \cite[Ch. XIV, 5.3.4]{pp2012} If $\alpha, \beta\in\mathcal{R}L$, then, for every $p\in\mathbb{Q}$, we have
	\[
			(\alpha+\beta)(p, -)=\bigvee_{q\in\mathbb{Q}}\alpha(q,-)\wedge\beta(p-q,-)
			\quad	\text{ and } \quad
			(\alpha+\beta)(-, p)=\bigvee_{q\in\mathbb{Q}}\alpha(-,q)\wedge\beta(-,p-q).
\]
	
			
		\begin{lemma}
			Suppose 
			$\alpha\in\mathcal{R}L$ such that $\coz(\alpha-{\bf r})=1$, where $0<r<\frac{1}{2}$. Then,  $\alpha$ is clean.
		\end{lemma}
		\begin{proof} By Lemma \ref{al},
		 $c= \alpha(-,r)$ is a complemented element in $L$ with $c^*=\alpha(r,-)$. Let us define 
		\[
			\varphi(p,q)=[c\vee\alpha(p,q)]\wedge[c\longrightarrow\alpha(p+1,q+1)]   
			\]
			for every $p,q\in\mathbb{Q}$.       Then, by {\bf B}, we have $\varphi\in \mathcal{R}L$. Now, since
			\[\begin{aligned}
				\varphi(r,-)&=[c\vee\alpha(r,-)]\wedge[c\longrightarrow\alpha(r+1,-)]\\&=(c\vee c^*)\wedge[c\longrightarrow\alpha(r+1,-)]\\&=1\wedge[c\longrightarrow\alpha(r+1,-)]\\
			&=c\longrightarrow\alpha(r+1,-)
			\end{aligned}\]
			and
			\[\begin{aligned}
				\varphi(-,-r)&=[c\vee\alpha(-,-r)]\wedge[c\longrightarrow\alpha(-,-r+1)]\\
				&=c\wedge[c\longrightarrow\alpha(-,-r+1)]\qquad {since} \qquad\alpha(-,-r) \leq c\\
				&=c\wedge \alpha(-,-r+1) \qquad\quad {by\quad{\bf A}} \\
				&=c,\qquad\quad {since} \qquad c\leq\alpha(-,-r+1)
			\end{aligned}	\]
			we have 
			\[\begin{aligned}
				\varphi(-,-r)\vee\varphi(r,-)&=c\vee [c\longrightarrow\alpha(r+1,-)]\\
				&\geq c\vee c^*\\&=1 \qquad {since} \qquad c^*\leq c\longrightarrow\alpha(r+1,-).
			\end{aligned}\]
			Therefore,  $\varphi$ is a unit element in $  \mathcal{R}L$. Now, by {\bf C}, for every $p\in\mathbb{Q}$, we have
			\[
			\begin{aligned}
				(e_c+\varphi)(p, -)&=\bigvee_{q\in\mathbb{Q}}e_c(q,-)\wedge\varphi(p-q,-)\\
				&=\bigvee_{q<0}1\wedge\varphi(p-q,-)\vee\bigvee_{0\leq q<1}c\wedge\varphi(p-q,-)\\
				&\vee\bigvee_{1\leq q\in\mathbb{Q}}0\wedge\varphi(p-q,-)\\
				&=\varphi(p,-)\vee\big(c\wedge\varphi(p-1,-)\big)\\
				&=\Bigg(\big(c\vee\alpha(p,-)\big)\wedge\big(c\longrightarrow\alpha(p+1, -)\big)\Bigg)\vee\\
				&\quad \Bigg(c\wedge(c\vee\alpha(p-1,-))\wedge(c\longrightarrow\alpha(p,-))\Bigg)\\
				&=\Bigg(\bigg(c\wedge\alpha(p+1, -)\bigg)\vee\bigg(\alpha(p,-)\wedge\big(c\longrightarrow\alpha(p+1, -)\big)\bigg)\Bigg)\vee\Bigg(c\wedge\alpha(p,-)\Bigg)\\
				& \leq \alpha(p+1,-)\vee \alpha(p,-)
			=\alpha(p,-)
			\end{aligned}
			\]
			and
			\[
			\begin{aligned}
				(e_c+\varphi)(-, p)&=\bigvee_{q\in\mathbb{Q}}e_c(-,q)\wedge\varphi(-,p-q)\\
				&=\bigvee_{q>1}1\wedge\varphi(-,p-q)\vee\bigvee_{0<q\leq1}c^*\wedge\varphi(-,p-q)\vee\bigvee_{0\geq q \in\mathbb{Q}}0\wedge\varphi(-,p-q)\\
				&=\varphi(-,p-1)\vee\big(c^*\wedge\varphi(-,p)\big)\\
				&=\Bigg(\big(c\vee\alpha(-,p-1)\big)\wedge\big(c\longrightarrow\alpha(-, p)\big)\Bigg)\vee\\
				&\quad \Bigg(c^*\wedge\bigg(\big(c\vee\alpha(-,p)\big)\wedge\big(c\longrightarrow\alpha(-,p+1)\bigg)\Bigg)\\
				&=\Bigg(\bigg(c\wedge\alpha(-, p)\bigg)\vee\bigg(\alpha(-,p-1)\wedge\big(c\longrightarrow\alpha(-, p)\big)\bigg)\Bigg)\vee\\
				&\qquad\Bigg((c^*\wedge c)\vee \big(c^*\wedge\alpha(-,p)\big)\wedge\big(c\longrightarrow\alpha(-,p+1)\big)\Bigg)\\
				& \leq \alpha(-,p)\vee\alpha(-,p-1)\vee \alpha(-,p)
				=\alpha(-,p).
			\end{aligned}
			\]
			Thus,  we can conclude that  $(e_c+\varphi)(p, q)\leq\alpha(p,q)$ for every $p,q\in\mathbb{Q}$. This implies that $\alpha=e_c+\varphi$ since $\mathcal{L}(\mathbb{R})$ is regular. Consequently, $\alpha$ is clean.
		\end{proof}
		
		Recall from \cite[Lemma 3.2]{g} that 
		 $\delta\in\mathcal{R}_c^*L$ is a unit of $\mathcal{R}_c^*L$  iff there exists $0<p\in\mathbb{Q}$ such that $\delta(-.-p)\vee\delta(p,-)=1$. In the proof of the previous lemma, it is evident that if $\alpha\in\mathcal{R}_cL$ ($\alpha\in\mathcal{R}_c^*L$), then $\varphi$ is a unit in  $\mathcal{R}_cL$ ($\mathcal{R}_c^*L$). Using this observation, the next theorem becomes self-evident. 
		\begin{theorem}\label{ac}
			The rings $\mathcal{R}_cL$  and $\mathcal{R}_c^*L$ are always clean for any frame $L$. 
		\end{theorem}
		
		An ideal $A$ of a commutative ring $R$ with $1$ is called a \emph{  $z$-ideal }if whenever two elements of $R$ are in the same set of maximal ideals and $A$ contains one of the elements, then it also contains the other (refer to \cite{mason}).
		Recall that an ideal $Q$ of $\mathcal{R}_cL$ is termed a \emph{  $z_c$-ideal} if, for any $\alpha\in\mathcal{R}_cL$ and $\beta\in Q$, the condition $\coz\alpha=\coz\beta$ implies $\alpha\in Q$ (see \cite{tes}).  
		It should be noted that an ideal $Q$ in $\mathcal{R}_cL$ is a $z_c$-ideal  iff it is a $z$-ideal (see \cite[Proposition 4.1]{a}). For any completely regular frame $L$, an ideal $H$ of $\mathcal{R}L$ is a $z$-ideal if, for any $\alpha\in\mathcal{R}L$ and $\beta\in H$, the condition $\coz\alpha=\coz\beta$ implies $\alpha\in H$ (see \cite{du}).
			Recall that if $S$ is a subring of a ring $R$ and $A$ is  an ideal of $R$, then $A\cap S$ forms an ideal of $S$ known as the \emph{  contraction} of $A$, denoted by $A^c$. 
		\begin {prop}
			For any completely regular frame $L$, an ideal $Q$ in $\mathcal{R}_cL$ is a $z_c$-ideal  iff it is a contraction of a $z$-ideal of $\mathcal{R}L$.
		\end {prop}
		\begin{proof}
			If $Q=H^c$, where $H$ is a $z$-ideal in $\mathcal{R}L$, then $Q$ is obviously a $z_c$-ideal in $\mathcal{R}_cL$. Conversely, let $Q$ be a $z_c$-ideal of $\mathcal{R}_cL$. Take
			$H=\{\delta\in\mathcal{R}L \mid \coz\delta\leq\coz\alpha \mbox{ for some } \alpha\in Q\}$.
			It is easy to verify that $H$ is a $z$-ideal in $\mathcal{R}L$, and $Q\subseteq H^c$. On the other hand, if $\delta\in H^c$, then there exists an $\alpha\in Q$ such that $\coz\delta\leq\coz\alpha$. 
			Since $Q$ is a $z_c$-ideal, $\alpha\delta\in Q$, and $\coz\alpha\delta=\coz\delta$, we can conclude that $\delta\in Q$, and thus the proof is complete.
		\end{proof}
		\begin {corol}\label{pc}
			For any completely regular frame $L$, an ideal in $\mathcal{R}_cL$ is a prime $z_c$-ideal  iff it is a contraction of a prime $z$-ideal of $\mathcal{R}L$.
		\end {corol}
		\begin{proof}
			To prove the nontrivial part of the corollary, let us assume that $P$  is a prime $z_c$-ideal of $\mathcal{R}_cL$. We can consider $S=\mathcal{R}_cL\setminus P$ as a multiplicatively closed set in 
			$\mathcal{R}L$. According to the previous proposition, we can choose a $z$-ideal $H$ in $\mathcal{R}L$ such that $P=H^c$. It is clear that
			$S\cap H=\emptyset$, which implies there exists a prime ideal $Q$ in $\mathcal{R}L$ minimal over $H$ 
			with $S\cap Q=\emptyset$. According to \cite[Theorem 1.1]{mason}, $Q$ is a $z$-ideal. 
			Subsequently, we can easily  show that  $P=H^c\subseteq Q^c\subseteq P$, which implies that $P=Q^c$, and thus we have completed the proof. 
		\end{proof}
		\begin {corol}\label{mc}
			For any completely regular frame $L$, every maximal ideal $Q$ of  $\mathcal{R}_cL$ is a contraction of a maximal ideal in $\mathcal{R}L$. Moreover, if $Q=M^c$, where $M$ is maximal ideal  in $\mathcal{R}L$, then $Q$ is real whenever $M$ is real.
		\end {corol}
		\begin{proof}	
			Assume $Q$ is a maximal ideal in $\mathcal{R}_cL$. As $Q$ is a $z_c$-ideal, by the previous proposition, we can select a $z$-ideal $H$ in $\mathcal{R}L$ with $Q=H^c$. So, there exists a maximal ideal $M$ in $\mathcal{R}L$ with $H\subseteq M$. Thus, $ Q=H^c\subseteq M^c$ implies that $Q=M^c$ since $Q$ is maximal. For the final part, if $M$ is real, then the monomorphism
		$			\Psi : \frac{\mathcal{R}_cL}{Q} \to \frac{\mathcal{R}L}{M}$,
			defined by $\Psi(\alpha+Q)=\alpha+M$ completes the proof.
		\end{proof}
		
		The initial part of the subsequent  lemma was presented in \cite[Proposition 10.2]{tes} but  without its accompanying proof. Here,  we show that the converse of this statement holds  for any zero-dimensional frame.
		\begin{lemma}\label{dc}
			Suppose $\alpha, \beta\in\mathcal{R}_cL$. Then $(\coz\alpha)^*\leq(\coz\beta)^*$ implies that $\ann(\alpha)\subseteq\ann(\beta)$. The converse holds if $L$ is zero-dimensional.
		\end{lemma}
		\begin{proof}
			By applying a $\mathcal{R}_cL$ version of the proof  of Lemma 4.1 in \cite {du0}, we can  show that if $(\coz\alpha)^*\leq(\coz\beta)^*$, then it follows that $\ann(\alpha)\subseteq\ann(\beta)$.
			To establish the next part, we aim to show that $(\coz\alpha)^*\wedge\coz\beta=0$, which is equivalent to showing that $(\coz\alpha)^*\leq(\coz\beta)^*$. We can start by selecting $a\in BL$ such that $a\leq(\coz\alpha)^*$. 	Then, it follows that $0=a\wedge(\coz\alpha)^{**}=\coz e_a\wedge(\coz\alpha)^{**}$, which  implies that $\coz e_a\alpha=\coz e_a\wedge\coz\alpha\leq\coz e_a\wedge(\coz\alpha)^{**}=0$.  As a result, we have $  e_a\alpha={\bf0}$, leading to $e_a\in\ann(\alpha)$, and thus $e_a\beta={\bf0}$, based on the current hypothesis. Thus, $a\wedge\coz\beta=\coz e_a\wedge\coz\beta=0$. So, the property of zero-dimensionality results in $(\coz\alpha)^*\wedge\coz\beta=0$.
		\end{proof}
		
		\begin {corol}
			Suppose  $L$ is a zero-dimensional frame. Then, any $\alpha\in\mathcal{R}_cL$ is a zero-divisor  iff $\coz\alpha$ is not a dense element of $L$.
		\end {corol}
		
		\begin{definition}
			An ideal $Q$ of $\mathcal{R}_cL$ is termed a $d_c$-ideal if, for every $\alpha\in\mathcal{R}_cL$ and $\beta\in Q$, the condition $(\coz\alpha)^*=(\coz\beta)^*$ implies that $\alpha\in Q$.
		\end{definition}
		The following are characterizations of $d_c$-ideals in terms of the cozero elements.
		\begin {prop}\label{d}
			The following are equivalent for every ideal  $Q$ in the ring $\mathcal{R}_cL$.
			\begin{enumerate} 
				\item $Q$ is a $d_c$-ideal.
				\item For any $\alpha, \beta\in\mathcal{R}L$, if $\alpha\in Q$ and $(\coz\alpha)^* \leq (\coz \beta)^*$, then $\beta\in Q$.
				\item For any $\alpha, \beta\in\mathcal{R}L$, if $\alpha\in Q$ and $\coz \beta \leq (\coz\alpha)^{**}$, then $\beta\in Q$.
			\end{enumerate}
		\end {prop}
		\begin{proof}
			In view of the fact that, $(a\wedge b)^*= b^*$ whenever $a^*\leq b^*$ for any $a, b\in L$, it is easy to show that (1) implies (2). It is clear that (2)$\Rightarrow $(3) $\Rightarrow$ (1).
		\end{proof}
		
		For any $\alpha\in\mathcal{R}_cL$, let $P_\alpha$ denote the intersection of all minimal prime ideals of $\mathcal{R}_cL$ that contain $\alpha$. A proper ideal $Q$ of $\mathcal{R}_cL$
		is called a \emph{ $d$-ideal} if $P_\alpha\subseteq Q$ for every $\alpha\in Q$. Since the ring $\mathcal{R}_cL$ is reduced, according to \cite[Proposition 1.4]{a2000}, a proper ideal $Q$ of $\mathcal{R}_cL$ is a $d$-ideal  iff $\ann(\alpha)=\ann(\beta)$ and $\alpha$ belonging to $Q$ implies that $\beta$ also belongs to $Q$.
		The proof of the next theorem follows from Lemma \ref{dc}, Proposition \ref{d}, and \cite[Proposition 1.4]{a2000}.
		\begin{theorem}\label{d1}
			If $L$ is zero-dimensional, then an ideal  $Q$ of $\mathcal{R}_cL$ is a $d$-ideal  iff it is a $d_c$-ideal.
		\end{theorem}

\section{The ${\bf O}_c$-and ${\bf M}_c$-ideals of $\mathcal{R}_cL$} \label{4}

In this section, our goal is to provide examples of  specific types of $z_c$-ideals and $d_c$-ideals. We start with the following lemma.

\begin{lemma}\label{c1}
	The following statements are true for any frame $L$.
	\begin{enumerate} 
		\item Given $S\subseteq BL$ such that $a=\bigvee\limits_{s\in S} s$, then  $\bigvee r_0(a)=a$. In particular, for any $a\in L$, we have $\bigvee r_0(a)=a$ when $L$ is a zero-dimensional frame.
		\item  For any $s\in\Coz_cL$, $\bigvee r_0(s)=s$.
		\item For any $a\in L$, if $\bigvee r_0(a)=a$, then  $r_0(a^*)=\big(r_0(a)\big)^*$.
		\item For any zero-dimensional frame $L$, the map $r_0$ commutes with pseudocomplements.
		\item For any $a,b\in BL$, we have $r_0(a\vee b)=r_0(a)\vee r_0(b)$.
		\item  If $I\prec\!\!\prec J$ in $\beta_0L$, then there exists $b\in J$ such that $\bigvee I\prec\!\!\prec b$ in $L$.
		\item Let $s\in \Coz_cL$ and $J\in \beta_0L$. Then $r_0(s)\prec\!\!\prec J$ in $\beta_0L$  iff there is $b\in J$ with $s\prec\!\!\prec b$ in $L$.
		\item Suppose $a\in BL$ and $J\in \beta_0L$. Then $r_0(a)\prec\!\!\prec J$ in $\beta_0L$  iff  $a\in J$.
		\item Suppose $L$ is a zero-dimensional frame, $a\in L$, and $J\in \beta_0L$. Then $r_0(a)\prec\!\!\prec J$ in $\beta_0L$  iff there exists $b\in J$ such that $a\prec\!\!\prec b$ in $L$.
	\end{enumerate}
\end{lemma}
\begin{proof}
	(1) Since $s\in r_0(a)$ for any $s\in S$, we  have  $ a=\bigvee\limits_{s\in S}s\leq\bigvee r_0(a).$
	Thus, we obtain $\bigvee r_0(a)=a$ since it is evident that $\bigvee r_0(a)\leq a$.
	
	(2) Given that $s$ belongs to the set $\Coz_cL$, there exists  a collection $\{s_n\}_{n\in \mathbb{N}}\subseteq  BL$ such that $s=\bigvee\limits_{n\in \mathbb{N}}s_n$. Consequently, by (1), it follows that $\bigvee r_0(s)=s$.
	
	(3) Because $r_0$ preserves zero and arbitrary meets, we have 
	$r_0(a)\wedge r_0(a^*)=r_0(a\wedge a^*)=r_0(0_L)=\{0_L\}=0_{\beta_0L}$, 
	showing that $r_0(a^*)\leq\big(r_0(a)\big)^*$. This establishes the inclusion $\subseteq$. To prove the other inclusion, we first show that $\bigvee\big(r_0(a)\big)^*\leq a^*$. If $x\in \big(r_0(a)\big)^*$, then the fact
	\[
	(r_0(a)\big)^*=\big\{c\in BL \mid c\wedge y=0_L, \hspace{3mm} \mbox{ for all } y\in r_0(a)\big\}
	\] 
	implies $x\wedge a=x\wedge\bigvee r_0(a)=\bigvee\big\{ x\wedge y \mid y\in r_0(a) \big\}=\bigvee\{0_L\}=0_L$.
	Thus $a\wedge \bigvee\big(r_0(a)\big)^*=\bigvee\big\{ a\wedge x \mid x\in \big(r_0(a)\big)^*\big\}=0_L$,
	implying $\bigvee \big(r_{0}(a)\big)^*\leq a^*$. Consequently, since $r_0$ is a right adjoint to $\bigvee$, we can conclude that $\big(r_0(a)\big)^*\leq r_0(a^*)$, proving the other inclusion.
	
	(4) It follows from (3) because  $a=\bigvee r_{0}(a)$ for any zero-dimensional frame.
	
	(5) Suppose $x\in r_0(a\vee b)$. Then $x=(x\wedge a)\vee(x\wedge b)$. This implies that $x\in r_0(a)\vee r_0(b)$, resulting in $r_0(a\vee b)\subseteq r_0(a)\vee r_0(b)$. The reverse inclusion is immediate,  so $$r_0(a\vee b)= r_0(a)\vee r_0(b).$$
	
	(6) Suppose $H\in \beta_0L$ such that $I\wedge H= 0_{\beta_0L}$ and $H\vee J=1_{\beta_0L}$. Then take $a \in H$ and $b\in J$ such that $a\vee b=1_L$. Hence,
	\begin{gather*}
		\begin{split}
			\bigvee I&=1_L\wedge\bigvee I
		=(a\vee b)\wedge\bigvee I
			 = (a\wedge\bigvee I)\vee (b\wedge\bigvee I)\\
			 &=\bigvee\{a\wedge x : x\in I\}\vee (b\wedge\bigvee I)
			 = 0_L\vee (b\wedge\bigvee I)
			 = b\wedge\bigvee I,
		\end{split}
	\end{gather*}
	which implies $\bigvee I\leq b$.  This shows that $\bigvee I\prec\!\!\prec b$ since $b\in BL$.
	
	(7) Suppose $b\in J$ such that $s\prec\!\!\prec b$. Since $b\in BL$, by (5), we  have $r_0(b)\prec\!\!\prec r_0(b)$ in $\beta_0L$. It follows that  $r_0(s)\prec\!\!\prec r_0(b)$ in $\beta_0L$ since $r_0(s)\subseteq r_0(b)$. Therefore, $r_0(b)\subseteq J$ implies $r_0(s)\prec\!\!\prec J$ in $\beta_0L$. The converse is evident from (2) and (6).
	
	(8) and (9) Obvious.
\end{proof}
The right adjoint $r: L\to\beta L$ preserves binary joints of the cozero elements, meaning that if $a, b\in\Coz L$, then $r(a\vee b)=r(a)\vee r(b)$. 
We want to show that if  $a, b\in\Coz_c L$, then $r_0(a\vee b)=r_0(a)\vee r_0(b)$.
We shall need the following notions introduced in \cite{ae}.
Suppose $a$ and $b$ are two elements of a frame $L$. Then  
$a$ is considered to be \emph{  $c$-completely below} $b$, denoted as $a\prec\!\!\prec_c\! b$ if there exists $\alpha\in\mathcal{R}_cL$ such that $a\wedge\coz\alpha=0$ and $\coz(\alpha-{\bf1})\leq b$.
Notably, the relations of $\prec$, $\prec\!\!\prec$, and $\prec\!\!\prec_c$ are equivalent for any compact zero-dimensional frame.
An ideal $I$ of $L$ is said to be \emph{  $c$-completely regular} if, for any $a\in I$, there exists $b\in I$ such that $a\prec\!\!\prec_cb$. The frame of $c$-completely regular ideals of a zero-dimensional $L$ is viewed as the $c$-Stone-\v{C}ech compactification of $L$, denoted by $\beta_c L$. We use $r_c$ to represent the right adjoint of the join map $j_c : \beta_c L \to L$.  Recall from \cite{ae} that $r_c(a)=\{x\in L \mid x\prec\!\!\prec_c a\}$  for any $a\in L$.
Define the map $g: \beta_0L\to\beta_c L$  by 
\[
g(I)=\{ x\in L \mid x\leq\bigvee J \mbox{ for some } J\in\beta_0L \mbox{ such that } J\prec I \mbox{ in } \beta_0L\}.
\]
That $g$ is an isomorphism follows from \cite[Page 110]{ba}. Moreover, if $J\in\beta_c L$ and we take
$
I=\bigvee\big\{H\in\beta_0L \mid \bigvee H\in J\big\},
$
then, we can deduce from \cite[Page 110]{ba} that $g(I)=J$. Now, it is easy to see that $r_c(a)=g(r_0(a))$ for any $a\in L$.
\begin {prop}\label{aa}
	Let $L$ be a zero-dimensional frame. Then $r_0(a\vee b)=r_0(a)\vee r_0(b)$ if $a,b\in\Coz_c L$.
\end {prop}
\begin{proof}
	By Lemma 5.5 in \cite{ae}, we have $r_c(a\vee b)=r_c(a)\vee r_c(b)$ for any $a,b\in\Coz_c L$. It follows that
	\[
			r_0(a\vee b)=g^{-1}(r_c(a\vee b))=g^{-1}\big(r_c(a)\vee r_c(b)\big)=g^{-1}\big(r_c(a)\big)\vee g^{-1}\big(r_c(b)\big)=r_0(a)\vee r_0(b).
	\]
\end{proof}

Let us recall how the ${\bf O}$- and ${\bf M}$-ideals are defined as explained in \cite{du0}. Consider a frame map $h : \beta L \to M$. We can define  the ideals ${\bf M}^h$ and ${\bf O}^h$ of $\mathcal{R}L$ as follows:
$${\bf M}^h=\{\alpha\in\mathcal{R}L  \mid h(r(\coz\alpha))=0_M\} 	\quad	\text{ and } \quad {\bf O}^h= \{\alpha\in\mathcal{R}L  \mid h(r((\coz\alpha)^*))=1_M\}.$$
In particular, if we have $I\in\beta L$ and $h: \beta L \to \uparrow I$ is the homomorphism defined by the correspondence $J \mapsto J\vee I$, then ${\bf M}^h$ and ${\bf O}^h$ are denoted by ${\bf M}^I$ and ${\bf O}^I$, respectively. Consequently, based on the properties of $r$, we obtain
${\bf M}^I=\{\alpha\in\mathcal{R}_cL  \mid r(\coz\alpha)\subseteq I\}$ and 
$$ {\bf O}^I= \{\alpha\in\mathcal{R}L  \mid r(\coz\alpha)\prec\!\!\prec I\}=\{\alpha\in\mathcal{R}L  \mid \coz\alpha\in I\}.$$

Motivated by the  above discussion and the definition of ${\bf O}_c$- and ${\bf M}_c$-ideals of $C_c(X)$ given in \cite{akko}, we introduce the ${\bf O}_c$- and ${\bf M}_c$-ideals of $\mathcal{R}_cL$ as follows.

{\rm\begin{example}\label{mo}
		Let $h : \beta_0L \to M$ be a frame map. We define   ${\bf M}^h_c$ and ${\bf O}^h_c$ as follows:
		$${\bf M}^h_c=\{\alpha\in\mathcal{R}_cL  \mid h(r_0(\coz\alpha))=0_M\} 	\quad	\text{ and } \quad {\bf O}^h_c= \{\alpha\in\mathcal{R}_cL  \mid h(r_0((\coz\alpha)^*))=1_M\}.$$
		In particular, if we have $I\in\beta_0L$ and $h: \beta_0L \to \uparrow\!I$ is the homomorphism defined by the correspondence $J \mapsto J\vee I$, then we denote 
		${\bf M}^h_c$ and ${\bf O}^h_c$ by ${\bf M}^I_c$ and ${\bf O}^I_c$, respectively, so that 
		$${\bf M}^I_c=\{\alpha\in\mathcal{R}_cL  \mid r_0(\coz\alpha)\subseteq I\} 	\quad	\text{ and } \quad {\bf O}^I_c= \{\alpha\in\mathcal{R}_cL  \mid I\vee r_0((\coz\alpha)^*)=1_{\beta_0L}\}.$$
		By (2) and (3) of Lemma \ref{c1}  and the fact that the rather  below relation coincides with the completely below relation in $\beta_0L$, we would have
		${\bf O}^I_c= \{\alpha\in\mathcal{R}_cL  \mid  r_0(\coz\alpha)\prec\!\!\prec I\}$,
		and consequently, by Lemma \ref{c1} (7),
		${\bf O}^I_c= \{\alpha\in\mathcal{R}_cL  \mid \coz\alpha\prec\!\!\prec\coz\delta \in I \mbox{ for some } \delta\in\mathcal{R}_cL\}$.
		It is straightforward to show ${\bf M}^I_c$ is a $z_c$-ideal. 
		Using a proof analogous to Example 4.11 in \cite{du0}, we can establish that ${\bf O}^I_c$ is an ideal.
		Now if $\alpha\in {\bf O}^I_c$ and $\delta\in\mathcal{R}_cL$ such that $(\coz\alpha)^* = (\coz \beta)^*$, then $1_{M}=h(r_0((\coz\alpha)^*)) = h(r_0((\coz\beta)^*))$, showing that $\beta\in {\bf O}^I_c$. So, ${\bf O}^I_c$ is a $d_c$-ideal. It is evident that ${\bf O}^I_c\subseteq {\bf M}^I_c$ for any $I\in\beta_0L$. Clearly, the other inclusion is true whenever $I\in B(\beta_0L)$.
\end{example}} 

\begin{remark}		
	It is straightforward to observe that 	${\bf M}^I_c={\bf O}^I_c=\{\bf0\}$ whenever $h : \beta_0L\to M$ is a dense homomorphism.
	Conversely, if ${\bf O}^I_c$ is the zero ideal, then $h$ is dense. Let $I\in\beta_0L$ such that $h(I)=0_M$. Suppose that $\alpha\in\mathcal{R}_cL$ for which $\coz\alpha\in I$. Then, by Lemma \ref{c1} (7), $r_0(\coz\alpha)\prec\!\!\prec I$, 
	and so $h(r_0((\coz\alpha)^*))=1_M$. This means that $\alpha\in{\bf O}^I_c$. Therefore, the only cozero element of $\Coz_cL$ contained in $I$ is $0_L$. Consequently, $I=\{0_L\}=0_{\beta_0L}$.
\end{remark}
By using Lemma \ref{cb}, we obtain the following result, which helps  us to gain a clearer understanding of the elements in ${\bf O}^I_c$-ideals.

\begin{theorem}\label{ig}
	If $I\in\beta_0L$, then  the ideal ${\bf O}^I_c$ is generated by a set of idempotents  in $\mathcal{R}_cL$.
\end{theorem}
\begin{proof}
	We claim that ${\bf O}^I_c$ is generated by the $E=\{e_a \mid a\in I\}$. 
	If $a\in I$, then $\coz e_a\prec\!\!\prec\coz e_a=a\in I$ implies $e_a$ belongs to ${\bf O}^I_c$. So we have $E\subseteq {\bf O}^I_c$. Now, if $\alpha\in {\bf O}^I_c$, then there is an element $\delta\in\mathcal{R}_cL$  
	such that $\coz\delta\in I$ and $\coz\alpha\prec\!\!\prec\coz\delta$. Since $\coz\delta\in BL$, we have $\coz\delta=\coz e_{\coz\delta}$, and so, by Lemma \ref{cb}  there exists $\beta\in\mathcal{R}_cL$ such that $\alpha=\beta e_{\coz\delta}$, and we are done.
\end{proof}
We aim to emphasize that the preceding proof shows the next corollary, going beyond the mere assertion that ${\bf O}^I_c$ is generated only by idempotents. This result reveals a more substantial property.

\begin  {corol}\label{o}
	If $I\in Pt(\beta_0L) $ and $\alpha\in{\bf O}^I_c$, then there exists $\delta\in\mathcal{R}_cL$ such that $\alpha=e_{\coz\alpha}\delta$.
\end  {corol}

We conclude this section with the following lemma. To remind the reader, the next notation is used. 
If $Q$ is an ideal of $\mathcal{R}_cL$, then $\Coz[Q]$ represents the ideal of $\Coz_cL$ defined as $$\Coz[Q]=\{\coz\alpha \mid \alpha\in Q\}.$$
Conversely, if $J$ is an ideal of $\Coz_cL$, then $\Coz^\leftarrow[J] $ is the ideal of $\mathcal{R}_cL$ given by
$$\Coz^\leftarrow[J]=\{\alpha\in\mathcal{R}_cL \mid \coz\alpha\in J\}.$$
\begin{lemma}\label{i} The following statements hold for any frame $L$.
	\begin{enumerate} 
		\item If $I\in\beta_0L$, then $\bigvee\Coz[{\bf O}^I_c]=\bigvee\Coz[{\bf M}^I_c]=\bigvee I$.
		\item If $I, J\in\beta_0L$, then ${\bf M}^I_c= {\bf M}^J_c$  iff $I=J$.
	\end{enumerate}
\end{lemma}
\begin{proof}
	(1) If $\alpha\in{\bf M}^I_c$, then $r_0(\coz\alpha)\subseteq I$. So, according to Lemma \ref{c1}(3), 
	we  have $\coz\alpha=\bigvee r_0(\coz\alpha)\leq \bigvee I$. Consequently, $\bigvee\Coz[{\bf O}^I_c]\leq\bigvee\Coz[{\bf M}^I_c]\leq\bigvee I$. Now suppose $a\in I$. Then $e_a\in {\bf O}^I_c$, 
	and hence $a=\coz e_a\leq \bigvee\Coz[{\bf O}^I_c]$. It follows that $\bigvee I\leq\bigvee\Coz[{\bf O}^I_c]$, 
	and we are done.
	
	(2) To prove the nontrivial part of the lemma, assume ${\bf M}^I_c= {\bf M}^J_c$. Now, let $a\in I$. Then $r_0(\coz  e_a)\subseteq I$, showing that $  e_a\in{\bf M}^I_c= {\bf M}^J_c$, and so $r_0(\coz  e_a)\subseteq J$. This implies that   $r_0(a)\prec\!\!\prec r_0(\coz e_a)\subseteq J$, and so  $a\in J$ by lemma \ref{c1} (8). Consequently, $I\subseteq J$. Similarly, $J\subseteq I$.
\end{proof}

\section{ Describing maximal ideals in $\mathcal{R}_cL$}\label{5}

In this section, we use  previously obtained results to derive $\mathcal{R}_cL$ versions of well-known theorems concerning specific ideals in the ring $C_c(X)$.  One such theorem is the $\mathcal{R}_cL$ analogue of  the Gelfand and Kolmogoroff theorem, which concerns maximal ideals of $C_c(X)$ (refer to \cite[Theorems 4.2 and 4.8]{akko}. We start with the following lemma. By
using  a proof similar to \cite[Proposition 3.12]{eks}, we can show that the following statements hold:
\begin{enumerate} 
	\item If $J$ is a maximal ideal of $\Coz_c[L]$, then $\Coz^\leftarrow[J] $ is a maximal ideal of $\mathcal{R}_cL$.
	\item If $Q$ is a maximal ideal of $\mathcal{R}_cL$, then $\Coz[Q]$ is a maximal ideal of $\Coz_cL$.
	\item An ideal $Q$ is a $z_c$-ideal of $\mathcal{R}_cL$  iff $Q=\Coz^\leftarrow[\Coz[Q]]$.
\end{enumerate}

\begin{lemma}
	For any $I\in\beta_0L$, the ideal ${\bf M}^I_c$ is maximal  iff $I$ is a prime element of $\beta_0L$.
\end{lemma}
\begin{proof} 
	$(\Rightarrow)$ Suppose $J\in\beta_0L$ such that $I\subseteq J$. Then ${\bf M}^I_c\subseteq{\bf M}^J_c$, which implies that
	${\bf M}^I_c={\bf M}^J_c$ or ${\bf M}^I_c=\mathcal{R}_cL={\bf M}^{\beta_0L}_c$ by maximality. Now, by lemma \ref{i}, the latter implies $I=\beta_0L$ or the former implies $I=J$.  Therefore, $I$ is a maximal element of $\beta_0L$ as required.

	$(\Leftarrow)$ We must show that ${\bf M}^I_c$ is a maximal ideal in the case where $I\in Pt(\beta_0L)$. Since ${\bf M}^I_c$ is a $z_c$-ideal, it is sufficient, based on the above observations, 
	to prove $\Coz[{\bf M}^I_c]$ is a maximal ideal of $\Coz_cL$. To see this, if $s\in\Coz_cL$ such that $s\vee\coz\alpha\ne1_L$ for any $\alpha\in{\bf M}^I_c$, we must prove that $s\in\Coz[{\bf M}^I_c]$. Choose $\varphi\in\mathcal{R}_cL$ such that $s=\coz\varphi$. To find a contradiction, assume that $\coz\varphi\notin\Coz[{\bf M}^I_c]$.
	This assumption implies that $\varphi\notin{\bf M}^I_c$, and so $r_0(\coz\varphi)\not\subseteq I$, 
	implying that $r_0(\coz\varphi)\vee I=1_{\beta_0L}$ since $I$ is a maximal element. Therefore, there exists $a\in I$ such that $a\vee\coz\varphi=1_L$.  This implies that $s\vee\coz  e_a= s\vee a=\coz\varphi\vee a=1_L$, which is a contradiction since, by lemma \ref{c1} (8), we can conclude that  $e_a\in{\bf M}^I_c$.
\end{proof}

The upcoming theorem is the frame version of \cite[Theorem 4.2]{akko}.
\begin{theorem} \label{m}
	A subset $Q$ of $\mathcal{R}_cL$ is a maximal ideal  iff there exists a unique prime element $I$ of $\beta_0L$ such that $Q={\bf M}^I_c$.
\end{theorem}
\begin{proof} 
	The `only if'  part of the theorem can be readily derived from the previous lemma. To show the converse,
	assume that $Q$ is a maximal ideal of $\mathcal{R}_cL$. Let $J\in\beta_0L$ be defined as $J=\bigvee\{r_0(\coz\alpha) \mid \alpha\in Q\}$.
	If  $J=1_{\beta_0L}$, then  compactness of $\beta_0L$  would yield indices $n_1, \cdots, n_k$ such that 
	$r_0(\coz\alpha_{n_1})\vee\cdots\vee r_0(\coz\alpha_{n_k})=1_{\beta_0L}$. 
	For any $1\leq i\leq k$,  we have $r_0(\coz\alpha_{n_i})\leq r_0(\coz\alpha_{n_1}\vee\cdots\vee \coz\alpha_{n_k})$, which implies
\[
1_{\beta_0L}\leq r_0(\coz\alpha_{n_1}\vee\cdots\vee \coz\alpha_{n_k})=r_0(\coz\alpha_{n_1}^2\vee\cdots\vee \coz\alpha_{n_k}^2)=r_0\big(\coz\big(\alpha_{n_1}^2+\cdots\alpha_{n_k}^2\big)\big).
\]
	Putting $\delta=\alpha_{n_1}^2+\cdots\alpha_{n_k}^2$, we get $\delta\in Q$ for which $r_0(\coz\delta)=1_{\beta_0L}$. Thus, using part (2) of Lemma \ref{c1}, we obtain
	$\coz\delta=\bigvee r_0(\coz\delta)=1_L$. This demonstrates that the ideal $Q$ contains some invertible element, which is a contradiction. Therefore,  $J\ne 1_{\beta_0L}$, and so there exists a prime element $I$ of $\beta_0L$ such that $J\subseteq I$. If $\alpha\in Q$, then $r_0(\coz\alpha)\subseteq J\subseteq I$, which implies  $Q\subseteq {\bf M}^I_c$. The maximality of $Q$ implies that $Q={\bf M}^I_c$. Lastly,  Lemma \ref{i}(2) shows that such a $I$ is unique. 
\end{proof}	
Every maximal ideal of $\mathcal{R}L$ can be expressed as ${\bf M}^I$, where $I\in Pt(\beta L)$, while every maximal ideal of $\mathcal{R}_cL$ can be expressed as ${\bf M}_c^I$, where $I\in Pt(\beta_0L)$. With Corollary  \ref{mc} in mind, 
for any $I\in Pt(\beta_0L)$, there exists a point $\pi_I\in Pt(\beta L)$ such that ${\bf M}_c^I={\bf M}^{\pi_I}\cap\mathcal{R}_cL$. This implies that if $\delta\in\mathcal{R}_cL$, then $r_0(\coz\delta)\subseteq I$  iff $r(\coz\delta)\subseteq \pi_I$. It is important to note that if  the corresponding point $\pi_I\in Pt(\beta L)$ is not unique for some $I\in Pt(\beta_0L)$, we can always select a unique point $\pi_I\in Pt(\beta L)$ for each point $I\in Pt(\beta_0L)$. Consequently, we can represent  maximal ideals of $\mathcal{R}_cL$ as follows:

\begin{theorem}
	For any completely regular frame $L$, every maximal ideal of $\mathcal{R}_cL$ can be precisely represented as
	\[
	{\bf M}_c^I=\{\delta\in\mathcal{R}_cL \mid r(\coz\delta)\subseteq \pi_I \}, \qquad I\in Pt(\beta_0L).
	\]
	
\end{theorem}
 Now we present the frame version of McKnight's for $\mathcal{R}_cL$. 
\begin {prop}\label{m1}
	If $Q$ is an ideal of $\mathcal{R}_cL$, then there is $I\in\beta_0L$ such that ${\bf O}^I_c\subseteq Q\subseteq{\bf M}^I_c$.
\end {prop}
\begin{proof}
	Define $I$ as the supremum of the set $\{r_0(\coz\alpha) \mid \alpha\in Q\}$, that is, $I=\bigvee\{r_0(\coz\alpha) \mid \alpha\in Q\}$.
	Then it follows immediately that $Q\subseteq{\bf M}^I_c$. To show the other inclusion, suppose $\beta\in {\bf O}^I_c$.
	Take $\delta\in\mathcal{R}_cL$ such that $\coz\delta\in I$ and $\coz\beta\prec\!\!\prec\coz\delta$. This shows that there exists $\alpha_1, \ldots , \alpha_n\in Q$ such that $\coz\beta\prec\!\!\prec\coz\delta\leq\bigvee_{i=1}^n\coz\alpha_i=\coz\varphi$, where $\varphi=\sum_{i=1}^n\alpha_i^2\in Q$
	So, by Lemma \ref{cb}, there is $\rho\in\mathcal{R}_cL$ such that $\beta=\rho\varphi$. 
	This shows that $\beta$ is an element of $Q$ since $Q$ is an ideal of $\mathcal{R}_cL$.
\end{proof}

It is important to emphasize that the ideals ${\bf O}^I_c$ and ${\bf M}^I_c$ mentioned in the foregoing proposition are not claimed to be unique. If $P$ is a prime ideal of $\mathcal{R}_cL$, then it is found that there exists a unique point $I\in\beta_0L$ such that ${\bf O}^I_c\subseteq P\subseteq{\bf M}^I_c$. We will now introduce the next lemma.
\begin{lemma}\label{om}
	The following statements hold for any frame $L$.
	\begin{enumerate}
		\item If $I\in Pt(\beta_0L)$, then the only maximal ideal containing ${\bf O}^I_c$ is ${\bf M}^I_c$.
		\item Suppose $I\in Pt(\beta_0L)$ and $\alpha\in\mathcal{R}_cL$. Then $\alpha\in{\bf O}^I_c$  iff $\alpha\beta=\textbf{0}$ for some $\beta\notin\ {\bf M}^I_c$.
		\item An ideal $Q$ in $\mathcal{R}_cL$ is contained in a unique maximal ideal ${\bf M}^I_c$ for some $I\in Pt(\beta_0L)$  iff ${\bf O}^I_c\subseteq Q$.
	\end{enumerate}
\end{lemma}
\begin{proof}
	(1) It is clear that ${\bf O}^I_c\subseteq{\bf M}^I_c$. Suppose $Q$ is a maximal ideal of $\mathcal{R}_cL$ containing ${\bf O}^I_c$. Then, by Theorem \ref{m}, there exists a unique prime element $J$ of $\beta_0L$ such that $Q={\bf M}^J_c$. 
	We only need to demonstrate that $I=J$. Let $a\in I$. 
	Then $  e_a\in {\bf O}^I_c\subseteq{\bf M}^J_c$, showing that $r_0(a)=r_0(\coz  e_a)\subseteq J$. 
	Since $I=\bigvee\{r_0(x) \mid x\in I\}$, it follows that $I\subseteq J$, and so $I=J$ by maximality.
	
	(2) If  $\alpha\in{\bf O}^I_c$, then, by Corollary \ref{o}, there exists an idempotent $e\in{\bf O}^I_c$ and $\delta\in\mathcal{R}_cL$ such that $\alpha=e\delta$. Putting $\beta=e-{\bf 1}$, then we have $$\alpha\beta=e\alpha-\alpha=e^2\delta-\alpha=e\delta-\alpha=\alpha-\alpha={\bf0},$$
	and $e\in{\bf O}^I_c\subseteq{\bf M}^I_c$ implies that $\beta\notin{\bf M}^I_c$.
	
	Conversely, If $\alpha\beta={\bf0}$, then $\coz\alpha\wedge\coz\beta=0$. Since $\beta\notin{\bf M}^I_c$, we can conclude that $ r_0(\coz\beta)\not\subseteq I$. This in turn implies that $ r_0(\coz\beta)\vee I=1_{\beta_0L}$ because $I$ is a maximal element of $\beta_0L$. So, there exists $b\in I$ such that $\coz\beta\vee b=1$, 
	which shows that $(\coz\beta)^*\prec\!\!\prec b$. Consequently, we obtain $\coz\alpha\leq(\coz\beta)^*\prec\!\!\prec b\in I$, showing that $\alpha\in {\bf O}^I_c$.
	
	(3) Follows directly from (1) and (2).
\end{proof}

\begin {prop}
	Let $P$ be a prime ideal of $\mathcal{R}_cL$. Then there exists a unique prime element $I$ of $\beta_0L$ such that ${\bf O}^I_c\subseteq P\subseteq{\bf M}^I_c$.
\end {prop}
\begin{proof}
	Because any prime ideal is contained in some maximal ideal, Theorem \ref{m} shows that $P\subseteq {\bf M}^I_c$ for some prime element $I$ of $\beta_0L$. Our objective is to demonstrate that ${\bf O}^I_c\subseteq P$. If we take $\alpha\in {\bf O}^I_c$, then, by Lemma \ref{om}, there exists $\beta\in\mathcal{R}_cL$ such that $\beta\notin{\bf M}^I_c$ and $\alpha\beta=\textbf{0}$. 
	Therefore, $\alpha\in P$ since $P$ is a prime ideal and $\beta\notin P$. So, based on the preceding lemma,
	${\bf M}^I_c$ is the unique maximal ideal containing $P$.
\end{proof}
\section{Describing fixed maximal ideals in $\mathcal{R}_cL$}\label{6}
To begin, recall from \cite{du2} that an ideal $Q$ of $\mathcal{R}L$ or $\mathcal{R}^*L$ is defined to be \emph{  fixed} if $\bigvee\limits_{\alpha\in Q}\coz\alpha<1$. Based on this definition, we can  introduce the following definition.
\begin{definition}
	An ideal $Q$ of $\mathcal{R}_cL$ or $\mathcal{R}_c^*L$ is  fixed if $\bigvee\limits_{\alpha\in Q}\coz\alpha<1$.	
\end{definition}	

As proved in \cite[Proposition 3.3]{du2}, fixed maximal ideals of $\mathcal{R}L$ correspond one-to-one with the points of $L$, and they represent precisely the sets
\[
{\bf M}_p=\{\alpha\in\mathcal{R}L \mid \coz\alpha\leq p\}, \qquad p\in Pt(L).
\]
We will now clear that the situation in $\mathcal{R}_cL$ is comparable.
Let us present the following types of fixed ideals of $\mathcal{R}_cL$.

\begin{definition}
	For any $a\in L$ such that $a<1$, we define the subset ${\bf M}_{ca}$ of $\mathcal{R}_cL$ as
	\[
	{\bf M}_{ca}=\{\alpha\in\mathcal{R}_cL \mid \coz\alpha\leq a\}.
	\]
\end{definition}

Evidently, ${\bf M}_{ca}$ is an ideal, and indeed, ${\bf M}_{ca}$ is fixed. Additionally, for any zero-dimensional frame, ${\bf M}_{ca}={\bf M}^{r_0(a)}_c$. Consequently, since $r_0(a)=r_0(b)$ implies $a=b$, according to Lemma \ref{i}, it follows that ${\bf M}_{ca}={\bf M}_{cb}$  iff $a=b$.  The first part of Lemma \ref{i}  helps us identify all fixed maximal ideals of $\mathcal{R}_cL$ in the following manner:
\begin{lemma}\label{f}
	If $I$ belongs to $\beta_0L$, then the ideal ${\bf M}^I_c$ is fixed  iff $\bigvee I< 1$.
\end{lemma}
\begin{proof}
	For any $I\in\beta_0L$, by Lemma \ref{i}, we have  $\bigvee\Coz[{\bf M}^I_c]=\bigvee I$. It follows that the ideal  ${\bf M}^I_c$ is fixed  iff $\bigvee I< 1$.
\end{proof}
\begin{lemma} \label{p}
	For any zero-dimensional frame $L$, if $I\in Pt(\beta_0L)$ such that $\bigvee I< 1$, then $\bigvee I\in Pt(L)$.
\end{lemma}
\begin{proof}
	Suppose $a\wedge b=\bigvee I$, where $a,b\in L$. Since $r_0$ is a right adjoint to $\bigvee$, it follows that $I\subseteq r_0(a\wedge b)=r_0(a)\wedge r_0(b)$.
	The assumption $\bigvee I< 1$ leads to the conclusion that $r_0(a\wedge b)\ne1_{\beta_0L}$, and so, we can conclude that $I=r_0(a)\wedge r_0(b)$ due to the maximality of $I$. 
	Consequently, we have $I=r_0(a)$ and $I= r_0(b)$.  Therefore,  $a=\bigvee I$, and similarly, $b=\bigvee I$, showing the desired result.
\end{proof}
\begin{theorem}\label{ff}
	For any zero-dimensional frame $L$, the fixed maximal ideals of $\mathcal{R}_cL$ are precisely the ideals ${\bf M}_{cp}$, where $p\in Pt(L)$. 
	These ideals are distinct for different points in the frame.
\end{theorem}
\begin{proof}
	It is clear that ${\bf M}_{cp}$ is a fixed maximal ideal for any $p\in Pt(L)$.
	Conversely, suppose  $Q$ is a fixed maximal ideal of $\mathcal{R}_cL$. By using Theorem \ref{m} and Lemma \ref {f}, we can  conclude that $Q={\bf M}^I_c$ for some point $I\in\beta_0L$ where $\bigvee I< 1$.  Putting $p=\bigvee I$, then   $p\in Pt(L)$ by Lemma \ref{p}. Our objective is to demonstrate that ${\bf M}^I_c$  is equal to $M_{cp}$.
	If $\alpha\in{\bf M}^I_c$, then $r_0(\coz\alpha)\subseteq I$, which implies that $\coz\alpha=\bigvee r_0(\coz\alpha)\leq\bigvee I=p$. Consequently, we have ${\bf M}^I_c\subseteq M_{cp}$.  Since $M_{cp}$ is a proper ideal and ${\bf M}^I_c$ is a maximal ideal, we can conclude that $Q={\bf M}^I_c=M_{cp}$, thus completing the argument. Finally, according to Lemma \ref{i}(2), the latter part of the theorem is evident.
	
\end{proof}

If $1\ne a\in L$, then it is evident that $ {\bf M}_{ca}={\bf M}_{a}\cap\mathcal{R}_cL$, where ${\bf M}_{a}=\{\delta\in\mathcal{R}L \mid \coz\delta\leq a\}$.  By using this information along with  \cite[Proposition 4.5] {tes}, we can derive the following corollary from the previous theorem.

\begin  {corol}
	For any zero-dimensional frame $L$, the fixed maximal ideals of $\mathcal{R}_cL$ are exactly the sets $Ker(\mathcal{R}_c\xi)$ for the characters $\xi : L\to {\bf2}=\{1, 0\}$. They are different for different characters.
\end  {corol}
For any zero-dimensional frame $L$, these results provide a criterion for spatiality  in terms of  fixed maximal ideals of the ring $\mathcal{R}_cL$. As a reminder from Lemma 3.1 in \cite{du1}, a regular frame is spatial  iff every element is not the top is below some point of the frame. By applying this characterization, we can derive the upcoming corollary from the above results.
\begin  {corol}
	The following are equivalent for each zero-dimensional frame $L$.
	\begin{enumerate} 
		\item $L$ is spatial.
		\item Every fixed ideal of $\mathcal{R}_cL$ is contained in a fixed maximal ideal.
		\item  Every fixed ideal of $\mathcal{R}_cL$ is contained in $Ker(\mathcal{R}_c\xi)$ for some character $\xi : L\to {\bf2}=\{1, 0\}$.
	\end{enumerate}
\end  {corol}
\begin  {corol}
	For any zero-dimensional frame $L$, if $p\in Pt(L)$, then the ideal ${\bf M}_{p}$ is the only maximal ideal of $\mathcal{R}L$ for which ${\bf M}_{cp}={\bf M}_{p}\cap\mathcal{R}_cL$.
\end  {corol}
\begin{proof}
	Suppose $M\ne{\bf M}_{p}$ is a maximal ideal of $\mathcal{R}L$ with ${\bf M}_{cp}=M\cap\mathcal{R}_cL$, and we aim to find a contradiction. Let $\alpha\in M$ be such that $\coz\alpha\not\leq p$. Since $L$ is zero-dimensional, there exists an element $a\in BL$ satisfying $ a\leq\coz\alpha$ and $a\not\leq p$. Given that $M$ is a $z_c$-ideal and $\coz e_a=a\leq\coz\alpha$, we can conclude that $e_a\in M$. On the other hand,  since $a\not\leq p$, it follows that $e_a\not\in{\bf M}_{cp}$. Therefore,  $e_a\in M\cap\mathcal{R}_cL\setminus{\bf M}_{cp}$, which present a contradiction.
\end{proof}
As shown in \cite[Lemma 4.7]{du3}, completely regular frames possess a criterion for compactness  in terms of fixed maximal ideals of $\mathcal{R}L$. We will now establish that the situation for  zero-dimensional frames is similar.
\begin{theorem}\label{fm}
	The following are equivalent for a zero-dimensional frame $L$.
	\begin{enumerate} 
		\item $L$ is compact.
		\item Every proper ideal of $\mathcal{R}_cL$ is  fixed.
		\item  Every maximal ideal of $\mathcal{R}_cL$ is  fixed.
		\item  Every proper ideal of $\mathcal{R}_c^*L$ is  fixed.
		\item  Every maximal ideal of $\mathcal{R}_c^*L$ is  fixed.
	\end{enumerate}
\end{theorem}
\begin{proof}
	(1)$\Rightarrow$ (2) Suppose $Q$ is an ideal of $\mathcal{R}_cL$ such that $\bigvee\limits_{\alpha\in Q}\coz\alpha=1$. So compactness of $L$ implies that  $\bigvee\limits_{i=1}^n\coz\alpha_i=1$ for some $\alpha_1, \alpha_2,\cdots,\alpha_n\in Q$, which leads to
	\[
	1=\bigvee\limits_{i=1}^n\coz\alpha_i=\bigvee\limits_{i=1}^n\coz\alpha_i^2=\coz(\alpha_1^2+\cdots\alpha_n^2).
	\]
	Letting $\varphi=\alpha_1^2+\cdots\alpha_n^2$, then we  have $\varphi\in Q$ with $\coz\varphi=1$. So, ${\bf1}\in Q$, implying that
	$Q=\mathcal{R}_cL$, which establishes that (1) implies (2).
	
	(2)$\Rightarrow$ (1) Let us consider the  situation where $A\subseteq L$ and $\bigvee A=1$. Define 
	$J=\{ s\in\Coz_cL \mid s\leq a \mbox{ for some } a\in A\}$. 
	Then, by zero-dimensionaliy, we have $\bigvee J=1$. Next, let 
	$I=\{ t\in\Coz_cL \mid t\leq\bigvee S \mbox{ for some finite } S\subseteq J\}$. 
	It is clear that $I$ is an ideal of $\Coz_cL$ such that $\bigvee I=1$. Hence, $Q=\Coz^\leftarrow[I]$ is an ideal of $\mathcal{R}_cL$ with $\bigvee\limits_{\alpha\in Q}\coz\alpha=\bigvee I=1$. Based on the current hypothesis, it follows that $Q=\mathcal{R}_cL$, implying ${\bf1}\in Q$. As a result, $1\in I$, which in turn suggests that $\bigvee S=1$ for some finite  $S\subseteq J$. Hence, there exists a finite set $B\subseteq A$ such that $\bigvee B=1$, which establishes that $L$ is compact.
		
	(2)$\Rightarrow$ (4) If (2) holds, then by the equivalence of (1) and (2), we  can conclude that $L$ is compact. This, in turn,  implies that $\mathcal{R}_cL=\mathcal{R}_c^*L$, thereby establishing that  (4) also holds.
	
	(4)$\Rightarrow$ (2) Consider   a proper ideal $Q$ of $\mathcal{R}_cL$. Then, $Q^c=Q\cap\mathcal{R}_c^*L$ is a proper ideal of $\mathcal{R}_c^*L$.  Furthermore, for any $\delta\in Q$, the element
	$\delta^2(\bf1+\delta^2)^{-1}$ belongs to $Q^c$ with $\coz\delta=\coz\big(\delta^2(\bf1+\delta^2)^{-1}\big)$. This implies that 
	$\bigvee\limits_{\alpha\in Q}\coz\alpha=\bigvee\limits_{\beta\in Q^c}\coz\beta$. As a result, 
	the current hypothesis therefore leads to $\bigvee\limits_{\alpha\in Q}\coz\alpha<1$, showing that $Q$ is a fixed ideal.
	
	(2)$\Leftrightarrow$ (3) and (4)$\Leftrightarrow$ (5). These equivalence are a result of the fact every proper ideal of a commutative ring $R$ is contained in a maximal ideal of $R$.
\end{proof}
\section{Maximal ideals in $\mathcal{R}_c^*L$}\label{7}
In this section, all our frames are assumed to be completely regular.
The focus of this section is on characterizing the maximal ideals of the ring $\mathcal{R}_c^*L$. Recall from \cite{ball2002} that a quotient $f: L \rightarrow M$ is a $C$-\emph{quotient} (or a $C^*$-\emph{quotient}) if, for any $\varphi \in \mathcal{R}M$ ($\varphi \in \mathcal{R}^*M$), there exists $\bar{\varphi} \in \mathcal{R}L$ ($\bar{\varphi} \in \mathcal{R}^*L$) such that $f\bar{\varphi} = \varphi$. Moreover, the fact that  the join map $j_L: \beta L \to L$ is a $C^*$-quotient implies $\mathcal{R}(\beta{L}) \cong \mathcal{R}^*L$.

As described in \cite{du2}, we denote the ring isomorphism by $\mathfrak{t}_L$, defined as
\[
\mathfrak{t}_L: \mathcal{R}(\beta{L}) \rightarrow \mathcal{R}^*L \quad \alpha \mapsto j_L\alpha,
\]
and its inverse is denoted by $\alpha \mapsto \alpha^\beta$. In the article \cite{du2}, the author uses this isomorphism to describe the maximal ideals of the ring $\mathcal{R}^*L$. For any $I \in \beta{L}$ with $I < \top_{\beta{L}}$, the ideal $M^{*I}$ of $\mathcal{R}^*L$ is given by
$
M^{*I} = \{\alpha \in \mathcal{R}^*L \mid \coz (\alpha^\beta) \subseteq I\}.
$
It is shown that maximal ideals of $\mathcal{R}^*L$ are precisely the ideals $M^{*I}$, for a point $I \in \beta{L}$.


Recall from \cite{a} that a quotient $f: L \rightarrow M$ is a $C_c$-\emph{  quotient } (or a $C_c^*$-\emph{  quotient}) if, for any $\alpha\in\mathcal{R}_c(M)$ ($\alpha\in\mathcal{R}_c^*(M)$), there is $\bar{\alpha}\in\mathcal{R}_cL$ ($\bar{\alpha}\in\mathcal{R}_c^*L$) such that $f\bar{\alpha} = \alpha$.

The preceding discussion asks whether $L$ is a $C_c^*$-quotient, with the quotient $j_0 :\beta_0 L\to L$. It is evident that $L$ is not necessarily a $C_c$- quotient, for example,  take $L=\mathfrak{O}(\mathbb{N})$.
Considering the next proposition, the answer to this question is also negative.

\begin{lemma} \label{cs1}
	If  $j_0: \beta_0L \to L$  is a $C_c^*$-quotient, then $\mathcal{R}_c^*L$ is isomorphic to $\mathcal{R}_c(\beta_0L)$.
\end{lemma}
\begin{proof}
	Define 
	$\mathfrak{t}_0: \mathcal{R}_c(\beta_0{L})\rightarrow\mathcal{R}_c^*L$   by  $\mathfrak{t}_0(\alpha)=j_0\alpha$.
	It is easy to verify that $\mathfrak{t}_0$ is an isomorphism.
\end{proof}
\begin {prop}\label{qn}
	Suppose $L$ is a strongly zero-dimensional frame such that $\mathcal{R}(\beta L)\ne\mathcal{R}_c(\beta L)$ and $\mathcal{R}^*(L)=\mathcal{R}_c^*L$ (for example, $L=\mathfrak{O}(\mathbb{N})$). Then $j_0: \beta_0L \to L$  cannot be a $C_c^*$-quotient.
	In particular, if we exchange the strong zero-dimensionality of $L$ with the complete regularity of $L$, then $j_0: \beta_0L \to L$  cannot be a $C_c^*$-quotient.
\end {prop}
\begin{proof}
	To find a contradiction, assume that $j_0: \beta_0L \to L$ is a $C_c^*$-quotient. Then, by the foregoing lemma, $\mathcal{R}_c^*L\cong\mathcal{R}_c(\beta_0L)$. Since $L$ is strongly zero-dimensional, we can deduce from Theorem \ref{cs} that $\mathcal{R}_c(\beta L)\cong\mathcal{R}_c(\beta_0L)$.
	Therefore
$\mathcal{R}(\beta L)\cong\mathcal{R}^*(L)=\mathcal{R}_c^*L\cong\mathcal{R}_c(\beta_0L)\cong\mathcal{R}_c(\beta L)$.
	This implies that $\mathcal{R}(\beta L)=\mathcal{R}_c(\beta L)$, which is the desired contradiction.
	In the same way, we proved the first part, we can demonstrate the last part.
\end{proof}

We will now proceed to describe the maximal ideals of the ring $\mathcal{R}_c^*L$. 
To recall, an ideal $S$ in a $f$-ring $R$ is calld \emph{  absolutely convex} if, whenever $r, s\in R$ with $|r|\leq|s|$ and $s\in S$, then $r\in S$.
\begin {prop}\label{acc}
	If  $L$ is a frame, then an ideal $Q$ in $\mathcal{R}_c^*L$ is an absolutely convex ideal  iff it is a contraction of an absolutely convex ideal of $\mathcal{R}^*L$.
\end {prop}
\begin{proof}
	To prove the nontrivial part of the proposition, consider  an ideal $Q$  in $\mathcal{R}_c^*L$, and define
$
	H=\{\alpha\in\mathcal{R}^*L \mid |\alpha|\leq|\delta| \mbox{ for some } \delta \in Q\}.
$
	We first show that $H$ is an ideal of $\mathcal{R}^*L$.
	It is apparent that for any $\alpha, \beta \in H$, we have $\alpha+\beta\in H$. If $\alpha \in H$ and $\beta\in \mathcal{R}^*L$, then  there exist $n\in\mathbb{N}$ and $\delta\in Q$ such that $|\beta|\leq{\bf n}$ and $|\alpha|\leq|\delta|$. This leads to
	$|\alpha\beta|=|\alpha||\beta|\leq{\bf n}|\delta|=|{\bf n}\delta|$, implying that $\alpha\beta\in H$ since ${\bf n}\delta\in Q$. Consequently,
	$H$ is an ideal in $\mathcal{R}^*L$, and it is clearly a proper ideal if $Q$ is a proper ideal. It can be readily observed that $H$ is absolutely convex, contains $Q$, and satisfies $Q=H^c=H\cap\mathcal{R}^*L$.
\end{proof}

Considering the above proposition and the fact that prime ideals in $\mathcal{R}_c^*L$ are absolutely convex (refer to \cite[Proposition 8.1]{tes}), we can immediately infer the following corollary.
\begin  {corol}
	Every maximal ideal of $\mathcal{R}_c^*L$  is a contraction of a maximal ideal in $\mathcal{R}^*L$.
\end  {corol}

In the article \cite{tes}, it is shown that the rings $\mathcal{R}_c(\beta{L})$ and $\mathcal{R}_c^*L$ are isomorphic. We denote by $\mathfrak{t}_{cL}$ the ring  isomorphism 
$
\mathfrak{t}_{cL}=\mathfrak{t}_L|_{\mathcal{R}_c(\beta{L})}: \mathcal{R}_c(\beta{L})\rightarrow\mathcal{R}_c^*L$
 defined by $\mathfrak{t}_cL(\alpha)=j_L\alpha$, the inverse of which we will denote by $\alpha\mapsto\alpha^{\beta_0}$.

\begin{definition}
	For each $I\in\beta L$ with $I<1_{\beta L}$ , the ideal ${\bf M}_c^{*I}$ of $\mathcal{R}_c^*L$ is expressed as 
	\[{\bf M}_c^{*I}=\{\alpha\in\mathcal{R}_c^*L \mid \coz(\alpha^{\beta_0})\subseteq I\}.\]
\end{definition}
It is clear that ${\bf M}_c^{*I}={\bf M}^{*I}\cap\mathcal{R}_c^*L=\{\alpha\in\mathcal{R}_c^*L \mid \coz(\alpha^{\beta})\subseteq I\}$. Maximal ideals of $\mathcal{R}_c^*L$ are among these ideals.
We introduce the relation  $I\backsim J$ on $Pt (\beta L)$, defined by ${\bf M}^{*I}\cap\mathcal{R}_c^*L={\bf M}^{*J}\cap\mathcal{R}_c^*L$. It is evident that this defines an equivalence relation on $Pt (\beta L)$. Let us denote the set of equivalence classes as $P=\{[I] \mid I\in Pt (\beta L)\}$, where $[I]$ represents the  equivalence class of $I\in Pt (\beta L)$. Consequently, the maximal ideals of $\mathcal{R}_c^*L$ are indexed by the  equivalence classes of this relation on $Pt (\beta L)$. In simpler terms, the entire collection of maximal ideals of $\mathcal{R}_c^*L$  precisely corresponds to the set  
$ \{{\bf M}^{*I}\cap\mathcal{R}_c^*L \mid I\in P\}$.
Utilizing these results, we derive a representation for maximal ideals of $\mathcal{R}_c^*L$ as follows:
\begin{theorem}\label{tw}
	Maximal ideals of $\mathcal{R}_c^*L$ can be described as follow
	$${\bf M}^{*I}_c =\{\alpha\in\mathcal{R}_c^*L \mid \coz(\alpha^{\beta_0})\subseteq I\}, \qquad I\in P.$$
\end{theorem}

According to Theorem \ref{fm}, if $L$ is strongly zero-dimensional, then  the maximal ideals of $\mathcal{R}_c(\beta L)$ are all fixed. This means that these maximal ideals can be precisely described as ${\bf M}_{cI}=\{\alpha\in\mathcal{R}_c(\beta L) \mid \coz\alpha\subseteq I\}$, where $I\in Pt(\beta L)$, as shown in Theorem \ref{ff}.

\begin {prop}
	If $L$ is a strongly zero-dimensional, then maximal ideals of $\mathcal{R}_c^*L$ are precisely the ideals  ${\bf M}^{*I}_c$, for $I\in Pt(\beta L)$. They are distinct for distinct $I$.
\end {prop}
\begin{proof}
	Based on the preceding discussion, the maximal ideals of $\mathcal{R}_c(\beta L)$ correspond exactly to the ideals ${\bf M}_{cI}$, where $I\in Pt(\beta L)$. 
	So, it suffices to show that ${\bf M}^{*I}_c=\mathfrak{t}_{cL}[{\bf M}_{cI}]$ for any $I\in Pt(\beta L)$. 
	Additionally, due to the maximality of $ \mathfrak{t}_cL[{\bf M}_{cI}]$, it is enough to establish that $\mathfrak{t}_{cL}[{\bf M}_{cI}]\subseteq {\bf M}^{*I}_c$ because ${\bf M}^{*I}_c$ is a proper ideal.
	Suppose $\alpha\in{\bf M}_{cI}$, then $\coz\alpha\subseteq I$. Consequently, since $\alpha=\mathfrak{t}_{cL}^{-1}\big(\mathfrak{t}_{cL}(\alpha)\big)=(j_L\alpha)^{\beta_0}$, we have $\coz\big((j_L\alpha)^{\beta_0}\big)\subseteq I$. It follows $j_L\alpha\in{\bf M}^{*I}_c$, by definition. Thus, $\mathfrak{t}_{cL}(\alpha)\in{\bf M}^{*I}_c$, as needed. Furthermore, the latter part of the proposition is evident since  $I\ne J$ in $Pt(\beta L)$ implies ${\bf M}_{cI}\ne{\bf M}_{cI}$.
\end{proof}
We close this section by introducing the following types of fixed ideals of $\mathcal{R}_c^*L$ and describing fixed maximal ideals of $\mathcal{R}_c^*L$ in the case where $L$ is a strongly zero-dimensional frame.
\begin{definition}
	For any $a\in L$ with $a<1$, define the subset ${\bf M}_{ca}^*$ of $\mathcal{R}_c^*L$ by
	\[
	{\bf M}_{ca}^*=\{\alpha\in\mathcal{R}_c^*L \mid \coz\alpha\leq a\}.
	\]
\end{definition}
Clearly, ${\bf M}_{ca}^*$ is a fixed ideal of $\mathcal{R}_c^*L$, and indeed, ${\bf M}_{ca}^*={\bf M}_{ca}\cap \mathcal{R}_c^*L$. 
By applying a version of the proof from \cite [Propositions 3.9 and 3.12 ]{du2} in $\mathcal{R}_cL$, we derive a similar result to Theorem \ref{ff}.
\begin {prop}
	The following statements  hold for any strongly zero-dimensional frame $L$.
	\begin{enumerate}
		\item For any $I\in Pt(\beta_0L)$, the maximal ideal ${\bf M}^{*I}_c$ is fixed  iff $\bigvee I<1$.
		\item The fixed maximal ideals of $\mathcal{R}_c^*L$ are exactly the ideals ${\bf M}_{cp}^*$ for $p\in Pt(L)$. 
	\end{enumerate}
\end {prop}
\section{The frame of maximal ideals of $\mathcal{R}_cL$ }\label{8}

Let $\ma(R)$ denote the set of maximal ideals of a commutative ring $R$ with an identity, equipped with the Hall-Kernel topology. Recall that the topology of $\ma(R)$ corresponds exactly to the frame 
\[
\mathfrak{O}(\ma(R))=\bigl\{\mathcal{D}(Q) \mid Q \mbox{ is an ideal of } R\bigr\},
\]
where, $\mathcal{D}(Q)= \bigl\{M\in\ma(R) \mid Q\not\subseteq M\bigr\}$ for any ideal $Q$ of $R$. 
We use the abbreviate $\mathcal{D}(\langle a\rangle)$ as $\mathcal{D}(a)$ for the principal ideal $\langle a\rangle$.
It is a well-known fact that for any completely regular frame $L$,
$
\mathfrak{O}(\ma(\mathcal{R}L))\cong\beta L.
$
The objective here is to show that for any  zero-dimensional frame $L$,
$
\mathfrak{O}(\ma(\mathcal{R}_cL))\cong\beta_0 L,
$
similar to the spatial result that $\ma(C_c(X))$ is homeomorphic to $\beta_0X$, the Banaschewski compactification of a zero-dimensional space $X$, for any zero-dimensional space $X$ (see \cite{akko}). The proof begins with a lemma.
\begin{lemma}\label{ec}
	Suppose $L$ is a frame and $\alpha\in\mathcal{R}L$ such that $\coz(\alpha-{\bf r})=1$, where $0<r<1$. Then, there is an idempotent element $e\in\mathcal{R}L$ such that $\coz e\leq\coz\alpha$ and $\coz(e-{\bf1})\leq\coz(\alpha-{\bf1}))$.
\end{lemma}
\begin{proof}
	By Lemma \ref{al}, $a= \alpha(r,-)$ is a complemented element in $L$ with $a^*=\alpha(-,r)$.
	Now let us take $e=e_a$, and it follows that $\coz e=a\leq\coz\alpha$ and $\coz(e-{\bf1})=a^*\leq\coz(\alpha-{\bf1}))$.
\end{proof}

Note that, for any idempotent $e\in\mathcal{R}_cL$,  we have
$\mathcal{D}(e)\cap\mathcal{D}(e-{\bf1})=\emptyset$ and $\mathcal{D}(e)\cup\mathcal{D}(e-{\bf1})=\ma(\mathcal{R}_cL)$.
This implies that $\mathcal{D}(e)$ is a complemented element of $\mathfrak{O}(\ma(\mathcal{R}_cL))$.
\begin {prop}
	The frame $\mathfrak{O}(\ma(\mathcal{R}_cL))$ is zero-dimensional for any frame $L$.
\end {prop}
\begin{proof}
	
	Consider $M\in \ma(\mathcal{R}_cL)$ such that $M\in\mathcal{D}(Q)= \bigl\{M\in\ma(\mathcal{R}_cL) \mid Q\not\subseteq M\bigr\}$, where $Q$ is an ideal of $\mathcal{R}_cL$. Then we have $Q+M=\mathcal{R}_cL$, which implies that there exists a $\alpha\in Q$ with $\alpha-{\bf1}\in M$. 
	Since $\alpha\in\mathcal{R}_cL$ implies $R_\alpha$ is countable, we can find $r\in\mathbb{R}$ such that $0<r<1$ and $\coz(\alpha-{\bf r})=1$. By the above lemma, there is an idempotent $e\in \mathcal{R}_cL$ such that
	$\coz e\leq\coz\alpha$ and $\coz(e-{\bf1})\leq\coz(\alpha-{\bf1})$. Since $\coz e\prec\!\!\prec\coz e\leq\coz\alpha$ and $\coz(e-{\bf1})\prec\!\!\prec\coz(e-{\bf1})\leq\coz(\alpha-{\bf1})$, respectively, we can conclude that $e\in Q$ and $e-{\bf1}\in M$, by Lemma \ref{cb}. Therefore, $M\in\mathcal{D}(e)\subseteq \mathcal{D}(Q)$. Consequently, there exists a set $E$ of the idempotents in $\mathcal{R}_cL$ such that
$\mathcal{D}(Q)=\bigcup_{e\in E}\mathcal{D}(e)$, which completes the proof.
\end{proof}
For any ideal $Q$ of $\mathcal{R}_cL$, let $J_Q$ be the element of $\beta_0L$ defined by
$J_Q=\bigvee\{r_0(\coz\alpha) \mid \alpha\in Q\}$. Since $\{r_0(\coz\alpha) \mid \alpha\in Q\}$ is directed and $Q$ is an ideal, we can conclude that the join above is effectively a union. In other words, $J_Q=\bigcup\{r_0(\coz\alpha) \mid \alpha\in Q\}$.
\begin{lemma}\label{q}
	The following  statements are true for any frame $L$.
	\begin{enumerate}
		\item If $P$ and $Q$ are two ideals of $\mathcal{R}_cL$ with $P\subseteq Q$, then $J_P\leq J_Q$.
		\item If $P$ and $Q$ are two ideals of $\mathcal{R}_cL$, then $J_{P\cap Q}=J_Q\wedge J_P$.
		\item If $\{Q_\lambda\}_{\lambda\in\Lambda}$ is a collection of ideals of $\mathcal{R}_cL$, 
		then $J_{{\bigcup\limits_{\lambda\in\Lambda}} Q_{\lambda}}=\bigvee\limits_{\lambda\in\Lambda }{J_{Q_{\lambda}}}$.
		\item If $I\in\beta_0L$, then $J_{{\bf O}_c^I}=I$.
	\end{enumerate}
\end{lemma}
\begin{proof}
	(1) Obvious.
	
	(2) Given that \[\begin{aligned} 
		J_Q\wedge J_P&=\bigg(\bigvee_{\alpha\in P} r_0(\coz\alpha)\bigg) \wedge \bigg(\bigvee_{\beta\in Q} r_0(\coz\beta)\bigg)\\
	&=\bigvee_{\alpha\in P}\bigvee_{\beta\in Q}r_0(\coz\alpha)\wedge r_0(\coz\beta)
		= \bigvee_{\alpha\in P}\bigvee_{\beta\in Q}r_0(\coz\alpha\wedge\coz\beta)\\
	&= \bigvee_{\alpha\in P}\bigvee_{\beta\in Q}r_0(\coz\alpha\beta)
		\leq \bigvee_{\delta\in P\cap Q}r_0(\coz\delta),
	\end{aligned}
	\]
	we can infer that $J_Q\wedge J_P\leq J_{P\cap Q}$. On the other hand, since $P\cap Q$ is a subset of $P$ and $Q$, by Part (1), we have $J_{P\cap Q}\leq J_Q\wedge J_P$. So equality holds.
		
	(3) It is clear that 
$$
	J_{{\bigcup\limits_{\lambda\in\Lambda}} Q_{\lambda}}=\bigvee\limits_{\alpha\in{{\bigcup\limits_{\lambda\in\Lambda}} Q_{\lambda}}} r_0(\coz\alpha)=\bigvee\bigvee\limits_{\alpha\in Q_{\lambda}}r_0(\coz\alpha)=\bigvee\limits_{\lambda\in\Lambda }{J_{Q_{\lambda}}}.
$$
	
	(4) If  $\alpha\in{\bf O}_c^I$, then there exists $a\in I$ such that  $\coz\alpha\leq a$. This shows that $r_0(\coz\alpha)\leq I$, confirming that $J_{{\bf O}_c^I}\leq I$. Now, suppose $a\in I$. Since $e_a\in{\bf O}_c^I$, we have $I\leq J_{{\bf O}_c^I}$. Thus, we can conclude that $J_{{\bf O}_c^I}= I$, as required.
\end{proof}
Recall that if $\{Q_\lambda\}_{\lambda\in\Lambda}$ is a collection of ideals of $\mathcal{R}_cL$, then 
$\mathcal{D}\bigg(\bigcup\limits_{\lambda\in\Lambda} Q_{\lambda}\bigg)=\mathcal\bigcup_{\lambda\in\Lambda}\mathcal{D}(Q_{\lambda})$,
and for two ideals $P$ and $Q$ of $\mathcal{R}_cL$, we also have $\mathcal{D}(P)\cap\mathcal{D}(Q)=\mathcal{D}(P\cap Q)$ .
\begin{theorem}\label{b}
	The frame $\mathfrak{O}(\ma(\mathcal{R}_cL))$ is isomorphic to $\beta_0L$  for any zero-dimensional frame $L$. 
\end{theorem}
\begin{proof}
	We claim that the map $	\Psi : \mathfrak{O}(\ma(\mathcal{R}_cL))\to\beta_0L$ defined by $\Psi(\mathcal{D}(Q))=J_Q$	
	is well-defined and constitutes a frame isomorphism. Initially,
	we establish that it is  well-defined. To show this, it suffices to prove that if $P$ and $Q$ are ideals in $\mathcal{R}_cL$
	such that $\mathcal{D}(P)=\mathcal{D}(Q)$, then $J_P=J_Q$. For any $I\in Pt(\beta_0L)$,
	\[\begin{aligned}
		Q\subseteq{\bf M}_c^I&\Leftrightarrow r_0(\coz\alpha)\leq I\quad \mbox{ for each } \quad\alpha\in Q\\
		&\Leftrightarrow J_Q\leq I.
	\end{aligned}	
	\]
	Since $J_Q$ is the meet of points of  $\beta_0L$ above it, and because $\mathcal{D}(P)=\mathcal{D}(Q)$ precisely when $P$ and $Q$ are contained in the same maximal ideals, we deduce that $J_P=J_Q$, as needed.
	
	It is evident that  $\Psi(0_{\mathfrak{O}(\ma(\mathcal{R}_cL))})=\Psi(\emptyset)=\Psi(\mathcal{D}({\bf0}))=\{0_L\}=0_{\beta_0 L}$ and
	$\Psi(1_{\mathfrak{O}(\ma(\mathcal{R}_cL))})=\Psi(\mathcal{R}_cL)=BL=1_{\beta_0 L}$.
	If $P$ and $Q$ are ideals in $\mathcal{R}_cL$, then Part (2) of the preceding lemma shows that
$
	\Psi\bigg(\mathcal{D}(P)\cap\mathcal{D}(Q)\bigg)=\Psi(\mathcal{D}(P\cap Q))
	=J_{P\cap Q}= J_Q\wedge J_P=\mathcal{D}(P)\cap\mathcal{D}(Q).
$
	Thus,  $\Psi$ preserves meet and order.
		Subsequently, consider a collection of ideals $\{Q_\lambda\}_{\lambda\in\Lambda}$ in $\mathcal{R}_cL$. Then, by Part (3) of the above lemma, we have
	\[
	\Psi\bigg(\mathcal\bigcup_{\lambda\in\Lambda}\mathcal{D}(Q_{\lambda})\bigg)=\Psi\bigg(\mathcal{D}\Big(\bigcup\limits_{\lambda\in\Lambda} Q_{\lambda}\Big)\bigg)=J_{{\bigcup\limits_{\lambda\in\Lambda}} Q_{\lambda}}=\bigvee\limits_{\lambda\in\Lambda }{J_{Q_{\lambda}}}=\bigvee\limits_{\lambda\in\Lambda }\Psi\Big({\mathcal{D}(Q_{\lambda})}\Big),
	\]
	thereby showing that $\Psi$ preserves joins. So, $\Psi$ is a frame homomorphism. Clearly, $\Psi$ is dense, and thus  one-to-one since $\mathfrak{O}(\ma(\mathcal{R}_cL))$ and $\beta_0L$ are zero-dimensional. If $I\in\beta_0L$, then, according to Part (4) of Lemma \ref{q},  $\Psi(\mathcal{D}({\bf O}_c^I))=I$. Consequently, $\Psi$ is onto, establishing it as an isomorphism.
\end{proof}

If $L$ is a compact frame, it can be observed that $R_\alpha$ is finite for any $\alpha\in\mathcal{R}L$ (for more detail, refer to \cite{ese, tes} ). Utilizing this observation, we determine that $\mathcal{R}(\beta_0 L)=\mathcal{R}_c(\beta_0 L)$ when $L$ is a compact frame. By combining Theorem \ref{m} and Theorem \ref{ff} with Theorem \ref{fm}, we can conclude that the space of maximal ideals of $\mathcal{R}_cL$ is homeomorphic to the space of maximal ideals of $\mathcal{R}_c(\beta_0L)$. From these observations, we deduce the following corollary based on the above theorem.
\begin  {corol}\label{bb} For any  zero-dimensional frame $L$,
	we have the following equivalences.
	\[
	\beta_0L\cong\mathfrak{O}(\ma(\mathcal{R}_cL))\cong\mathfrak{O}(\ma(\mathcal{R}_c(\beta_0L)))\cong\mathfrak{O}(\ma(\mathcal{R}(\beta_0L))).
	\] 
\end  {corol}			
\begin{acknowledgement} 
	Thanks are due to Professor Ali Akbre Estaji for reading the paper and giving comments.
	\end{acknowledgement}

	\end{document}